\documentclass[11pt,letterpaper,leqno]{amsart}
\usepackage[dvipsnames]{xcolor}
\usepackage{amsmath}
\usepackage{amsfonts}
\usepackage{amssymb}
\usepackage{graphicx}
\usepackage{hyperref}

\hypersetup{
  pdftitle={Lpp conjecture},
  pdfauthor={Stefano Pigola and Daniele Valtorta and Giona Veronelli},
  pdfsubject={},
  pdfkeywords={},
  pdfpagelayout=SinglePage,
  pdfpagemode=UseOutlines,
  colorlinks,
  bookmarksopen,
  linkcolor=[rgb]{0,0,0.7},
  urlcolor=[rgb]{0,0,0.4},
  citecolor=[rgb]{0.4,0.1,0}
}

\usepackage{a4wide} 
\usepackage{mathrsfs} 
\usepackage{relsize} 

\theoremstyle{plain}
\newtheorem{theorem}{Theorem}

\newtheorem{conjecture}[theorem]{Conjecture}
\newtheorem{corollary}[theorem]{Corollary}
\newtheorem{definition}[theorem]{Definition}
\newtheorem{lemma}[theorem]{Lemma}

\newtheorem{proposition}[theorem]{Proposition}

\theoremstyle{definition}
\newtheorem{example}[theorem]{Example}

\newtheorem{remark}[theorem]{Remark}

\numberwithin{equation}{section}
\numberwithin{theorem}{section}

\newcommand{\vol}{\mathrm{vol}}

\newcommand{\dive}{\operatorname{div}}
\renewcommand{\div}{\operatorname{div}}

\newcommand{\Hess}{\operatorname{Hess}}

\newcommand{\rr}{\mathbb{R}}

\newcommand{\nn}{\mathbb{N}}

\newcommand{\riem}{\operatorname{Riem}}

\newcommand{\ric}{\operatorname{Ric}}

\newcommand{\supp}{\operatorname{supp}}
\newcommand{\tr}{\operatorname{trace}}

\newcommand{\dv}{\,\mathrm{dv}}
\newcommand{\dx}{\,\mathrm{d}x}

\newcommand{\sgn}{\operatorname{sgn}}

\def\XXint#1#2#3{{\setbox0=\hbox{$#1{#2#3}{\int}$}
     \vcenter{\hbox{$#2#3$}}\kern-.5\wd0}}

\newcommand{\inj}{\mathrm{inj}}

\newcommand{\diam}{\mathrm{diam}}
\newcommand{\dist}{\mathrm{dist}}

\renewcommand{\a}{\alpha}

\newcommand{\ve}{\varepsilon}

\renewcommand{\l}{\lambda}

\newcommand{\vp}{\varphi}

\newcommand{\CD}{\mathcal{D}}

\newcommand{\CH}{\mathcal{H}}
\newcommand{\CL}{\mathcal{L}}

\newcommand{\B}{\mathscr{B}}
\newcommand{\C}{\mathscr{C}}

\newcommand{\R}{\mathscr{R}}
\newcommand{\norm}[1]{\left\|#1\right\|}
\newcommand{\ps}[2]{g\left(#1,#2\right)}

\newcommand{\ton}[1]{\left(#1\right)}
\newcommand{\qua}[1]{\left[#1\right]}
\newcommand{\cur}[1]{\left\{#1\right\}}
\newcommand{\abs}[1]{\left|#1\right|}

\renewcommand{\B}[2]{B_{#1}\ton{#2}}

\renewcommand{\R}{\mathbb{R}}

\begin{document}

\title[Approximation, regularity and positivity preservation]{Approximation, regularity and positivity preservation on Riemannian manifolds}
\author{Stefano Pigola}
\address{Universit\`a degli Studi di Milano-Bicocca\\ Dipartimento di Matematica e Applicazioni \\ Via Cozzi 55, 20126 Milano - Italy - EU}
\email{stefano.pigola@unimib.it}
\author{Daniele Valtorta}
\address{Universit\`a degli Studi di Milano-Bicocca\\ Dipartimento di Matematica e Applicazioni \\ Via Cozzi 55, 20126 Milano - Italy - EU}
\email{daniele.valtorta@unimib.it}
\author{Giona Veronelli}
\address{Universit\`a degli Studi di Milano-Bicocca\\ Dipartimento di Matematica e Applicazioni \\ Via Cozzi 55, 20126 Milano - Italy - EU}
\email{giona.veronelli@unimib.it}
\date{\today}

\begin{abstract}
The paper focuses on the $L^{p}$-Positivity Preservation property ($L^{p}$-PP for short) on a Riemannian manifold $(M,g)$. It states that any $L^p$ function $u$ with $1<p<+\infty$, which solves $(-\Delta + 1)u\ge 0$ on $M$ in the sense of distributions must be non-negative. Our main result is that the $L^{p}$-PP holds if (the possibly incomplete) $M$ has a finite number of ends with respect to some compact domain, each of which is $q$-parabolic for some, possibly different, values $2p/(p-1) < q \leq +\infty$. When $p=2$, since $\infty$-parabolicity coincides with geodesic completeness, our result settles in the affirmative a conjecture by M. Braverman, O. Milatovic and M. Shubin in 2002. On the other hand, we also show that the $L^{p}$-PP is stable by removing from a complete manifold a possibly singular set with Hausdorff co-dimension strictly larger than $2p/(p-1)$ or with a uniform Minkowski-type upper estimate of order $2p/(p-1)$. The threshold value $2p/(p-1)$ is sharp as we show that when the Hausdorff co-dimension of the removed set is strictly smaller, then the $L^{p}$-PP fails. This gives a rather complete picture. The tools developed to carry out our investigations include smooth monotonic approximation and consequent regularity results for subharmonic distributions, a manifold version of the Brezis-Kato inequality, Liouville-type theorems in low regularity, removable singularities results for $L^{p}$-subharmonic distributions and a Frostman-type lemma.

Since the seminal works by T. Kato, the $L^{p}$-PP has been linked to the spectral theory of Schr\"odinger operators with singular potentials $\Delta - V$. Here we present some applications of the main results of this paper to the case where $V\in L^p_{loc}$, addressing the essential self-adjointness of the operator when $p=2$ and whether or not $C^\infty_c(M)$ is an operator core for $\Delta-V$ in $L^p$.

\end{abstract}

\maketitle
\tableofcontents

\section{Introduction and main results}

\subsection{Basic notation}

Let $(M,g)$ be a connected, possibly incomplete, $n$-dimensional Riemannian manifold, $n \geq 2$, endowed with its Riemannian measure $\dv$. Unless otherwise specified, integration will be always performed with respect to this measure. The Riemannian metric $g$ gives rise to the intrinsic distance $\dist(x,y)$ between a couple of points $x,y \in M$. The corresponding open metric ball centered at $o \in M$ and of radius $R>0$ is denoted by $B_{R}(o)$. The $\riem$ and $\ric$ symbols are used to denote, respectively, the Riemann and the Ricci curvature tensors of $(M,g)$. Finally, the Laplace-Beltrami operator of $(M,g)$ is denoted by $\Delta = \tr \Hess = \div \nabla$. We stress that we are using the sign convention according to which, on the real line, $\Delta = +\frac{d^{2}}{dx^{2}}$.\smallskip

This paper deals with sub-solutions of elliptic PDEs involving the Schr\"odinger operator
\[
\CL = \Delta - \l(x)
\]
where $\l(x)$ is a smooth function.

We say that $u \in L^{1}_{loc}(M)$ is a {\it distributional solution} of $\CL u \geq f  \in L^{1}_{loc}(M)$ if, for every $0 \leq \vp \in C^{\infty}_{c}(M)$,
\[
\int_{M}u \CL\vp \geq \int_{M}f \vp.
\]
Sometimes, we will call such a $u$ a {\it distributional subsolution} of the equation $\CL u = f$. The notion of {\it distributional supersolution} is defined by reversing the inequalities and we say that $u \in L^{1}_{loc}(M)$ is a {\it distributional solution} of $\CL u = f$ if it is a subsolution and a supersolution at the same time.\smallskip

In the presence of more local regularity of the function involved we can also speak of a weak solution of the same inequality. Namely,  $u \in W^{1,1}_{loc}(M)$ is a {\it weak solution} of $\CL u \geq f \in L^{1}_{loc}(M)$ if, for every $0 \leq \vp \in C^{\infty}_{c}(M)$, it holds
\[
- \int_{M}  g(\nabla u , \nabla \vp)    \geq \int_{M} (\l + f) \vp
\]
By a density argument, the inequality can be extended to test functions $0 \leq \vp \in W^{1,\infty}_{c}(M)$. If the regularity of $u$ is increased to $W^{1,2}_{loc}(M)$ then test functions can be taken in $W^{1,2}_{c}(M)$.

Finally, we need to recall that  a function $u \in W^{1,1}_{loc}(M)$ is a distributional solution of $\CL u \geq f \in L^{1}_{loc}$ if and only if it is a weak solution of the same inequality.

\subsection{The BMS conjecture}

This paper and its companion \cite{GPSV-preprint} originate from the unpublished preprint \cite{PV1} by two of the authors. The main goal of the present paper is to expand the investigation of the BMS conjecture on possibly incomplete Riemannian manifolds. In \cite{GPSV-preprint} a generalization of the BMS conjecture, that relies on a completely new notion of distributional subsolutions, is obtained in the much broader setting of complete metric measure spaces.

BMS conjecture was introduced in \cite[Appendix B]{BMS}, and it is concerned with the $L^p$-positivity preserving property for Riemannian manifolds. We recall the definition of this property by B. G\"uneysu, \cite{Gu-JGEA}:

\begin{definition}
Let $1 \leq p \leq +\infty$. The Riemannian manifold $(M,g)$ is said to be $L^{p}$-Positivity Preserving ($L^{p}$-PP for short) if the following implication holds true:
\begin{gather}\label{P}\tag{$L^p\mathrm{-PP}$}
\begin{cases}
(- \Delta + 1) u \geq 0 \text{ distributionally on M}\\
u \in L^{p}(M)
\end{cases}
\Longrightarrow
u \geq 0\text{ a.e. on }M.
\end{gather}
More generally, one can consider any family of functions $\C \subseteq L^{1}_{loc}(M)$ and say that $(M,g)$ is $\C$-Positivity Preserving if the above implication holds when $L^{p}(M)$ is replaced by $\C$.
\end{definition}

The following conjecture, motivated by the study of self-adjointness of covariant Schr\"odinger operators (see the discussions in \cite{Gu-survey, Gu-book} and Section \ref{subsec:core} below) was formulated by M. Braverman, O. Milatovic and M. Shubin in \cite[Appendix B]{BMS}.

\begin{conjecture}[BMS conjecture]\label{BMS-conjecture}
Assume that $(M,g)$ is geodesically complete.  Then $(M,g)$ is $L^{2}$-positivity preserving.
\end{conjecture}

The validity of the BMS conjecture has been verified under additional restrictions on the geometry of the complete Riemannian manifold $(M,g)$. More precisely:
\begin{itemize}
\item In the seminal paper \cite[p.140]{Ka},  T. Kato proved that $\rr^{n}$ is $L^{2}$-PP.
 \item  In \cite[Proposition B.2]{BMS} it is assumed  that $(M,g)$ has $C^{\infty}$-bounded geometry, i.e., it satisfies $\| \nabla^{(j)}\riem \|_{L^{\infty}} < +\infty$ for any $j\in \nn$ and $\inj(M)>0$.
 \item  In \cite{Gu-JGEA}, B. G\"uneysu showed that $\ric \geq 0$ is sufficient for Conjecture \ref{BMS-conjecture} to be true. Subsequently, in \cite[Theorem XIV.31]{Gu-book}, he proved that if $\ric \geq -K^{2}$ then $(M,g)$ is $L^{p}$-PP on the whole scale $p \in [1,+\infty]$.
 \item  In \cite{BS}, D. Bianchi and A.G. Setti observed that the BMS conjecture is true even if the Ricci curvature condition is relaxed to
 \[
 \ric \geq - C (1+ r(x))^{2}
 \]
where $r(x) = \dist(x,o)$ for some origin $o \in M$. Under the same curvature assumptions, the $L^p$-PP can be extended almost directly to any $p\in[2,\infty)$, and with a little more effort to any $p\in[1,+\infty]$, \cite{Gu-book,MV}.
 \item  In the very recent \cite{MV}, L. Marini and the third author considered the case of a Cartan-Hadamard manifold (complete, simply connected with $\riem \leq 0$). In this setting it is proved that $(M,g)$ is $L^{p}$-PP for any $p \in [2,+\infty)$ provided
 \[
 -(m-1)B^{2}(1+r(x))^{\a +2} \leq \ric \leq -(m-1)^{2}A^{2}(1+ r(x))^{\a}
 \]
 for some $\a>0$ and  $B > \sqrt{2}(m-1)A>0$.
\end{itemize}

Kato's argument in $\rr^n$ relies on the positivity of the operator $(-\Delta + 1)^{-1}$ acting on the space of tempered distributions, which in turn is proved using the explicit expression of its kernel. Instead, in all the above quoted works on Riemannian manifolds, the proofs stem from an argument by B. Davies, \cite[Proposition B.3]{BMS} that relies on the existence of good cut-off functions with controlled gradient and Laplacian. Obviously, the construction of these cut-offs requires some assumption on the curvature.

\subsection{Riemannian manifolds and \texorpdfstring{$L^{p}$}{Lp}-Positivity Preservation} 

In this paper we prove some results on the $L^p$ positivity-preserving property for a Riemannian manifold $M$ and its link to completeness of $M$, the $p$-parabolicity of $M$, and the validity of the property on $M\setminus K$ in relation to the size of $K$.

The approach we use is somewhat different from the other results available in literature. In particular, our approach is based on a new \textsl{a priori} regularity result for positive subharmonic distributions (see Section \ref{sect:reg}), which permits to prove a Liouville type theorem and on a Brezis-Kato inequality on Riemannian manifolds (see Section \ref{sec:Kato}). Both these results, in turn, rely on a smooth monotonic approximation of distributional solutions of $\CL u \geq 0$ (see Section \ref{sect:approx}). This approximation can be proved using general potential-theoretic arguments, but in this paper we also present an explicit construction that uses the Riemannian Green's function as a sort of mollifier to smoothen any $L^1$ subsolution $u$.

This approach avoids any curvature restrictions on the manifold $M$, and allows us to prove that

\begin{theorem}\label{th:BMS}
 Let $(M,g)$ be a complete Riemannian manifold. Then $M$ is $L^{p}$-Positivity Preserving for every $p \in (1,+\infty)$. In particular, the BMS conjecture is true.
\end{theorem}

As a matter of fact, a by-product of our approach is that a complete manifold is $\C$-Positivity Preserving where
\[
\C = \{ u \in L^{1}_{loc}(M): \| u \|_{L^{p}(B_{2R}\setminus B_{R})} = o(R^{2/p}) \text{ as }R \to +\infty \},
\]
where $p\in (1,+\infty)$ (see Remark \ref{rem-subquadratic}).\smallskip

\begin{remark}
 We note explicitly that, in the statement of Theorem \ref{th:BMS}, the endpoint cases $p=1$ and $p=+\infty$ are excluded. This is because, in general, the corresponding property may fail. Indeed, it is well known that there are complete Riemannian manifolds which are not $L^{\infty}$-Positivity Preserving (since stochastically incomplete). On the other hand, \cite{BM} contains an example of complete manifold whose sectional curvature decays more than quadratically to $-\infty$ and such that the $L^{1}$-Positivity Preservation is not satisfied.
\end{remark}

\subsection{\texorpdfstring{$L^{p}$}{Lp}-Positivity Preservation, parabolicity, capacity and removable sets}
A natural problem is to understand to which extent geodesic completeness is a necessary condition for the \eqref{P} property to hold. We present two families of results in this direction: 
\begin{enumerate}
 \item Theorem \ref{th:caccio} states that a not necessarily complete Riemannian manifold $M$ still enjoys the \eqref{P} property if it has a finite number of ends, all of which are $q$-parabolic for possibly different values of $q>\frac{2p}{p-1}$.
 \item Corollary \ref{coro:LpPP cap}, Proposition \ref{prop_Mink} and Section \ref{sec_Kbig} deal with manifolds of the form $M\setminus K$, where $M$ is a complete Riemannian manifold and $K$ is a compact subset with either capacity or Minkowski-type upper bounds. 
\end{enumerate}
We refer the reader to \cite{mattilone, dl_book_tangent, grig} for a background on capacity and parabolicity. The relevant definitions will be also recalled in Section \ref{sec:parabolic}.

The main idea for these results is that sets that are sufficiently ``small'' in a suitable sense do not influence the behavior of solutions to $(-\Delta+1)u\geq 0$, and parabolicity of a manifold guarantees the same property for the boundary at infinity of $M$. 
\smallskip

To be more precise, we prove the following results:

\begin{proposition}\label{prop:LpPP cap}
Let $1<p<+\infty$ and let $M = N \setminus K$, where $(N,h)$ is an $n$-dimensional complete Riemannian manifold and $K\subset N$ is a compact set. Suppose that the Hausdorff dimension of $K$ satisfies
\begin{gather}
\dim_{H}(K)<n-\frac{2p}{p-1}\, .
\end{gather}
Then $M$ is $L^p$-Positivity Preserving.
\end{proposition}
As a matter of fact the conclusion of  Proposition \ref{prop:LpPP cap} holds if $K$ is  $q$-polar for some $\frac{2p}{p-1}<q \leq +\infty$, a condition that is implied by the smallness of Hausdorff dimension.

With a stronger assumption on the size of $K$, it is possible to deal with the threshold dimension as well. In particular, we have the following result. Here, and in the subsequent parts of the paper, given a subset $E$ of a complete Riemannian manifold $(N,h)$, we denote by $B_{r}(E) = \cup_{x\in E}B_{r}(x)$ its open tubular neighborhood.
\begin{proposition}\label{prop:LpPPMink}
 Let $1<p<+\infty$ and let $M = N \setminus K$, where $(N,h)$ is an $n$-dimensional complete Riemannian manifold and $K\subset N$ is a bounded set. If the tubular neighborhoods of $K$ have uniform volume bounds of the form
 \begin{gather}
  \vol\ton{\B r K }\leq C r^{\frac{2p}{p-1}}
 \end{gather}
for some $C$ independent of $r$, then $M$ is $L^p$-Positivity Preserving.
\end{proposition}

On the other hand, it is possible to build explicit examples of complete manifolds $(M,h)$ that lose the \eqref{P} property when a set $K$ of ``big size'' is removed from them. In particular, Proposition \ref{prop_no_good} shows that if the set $K$ has Hausdorff dimension 
\begin{gather}
\dim_{H}(K)>n-\frac{2p}{p-1}\, ,
\end{gather}
then the \eqref{P} property does not hold on $M\setminus K$, so that the dimensional threshold $n-\frac{2p}{p-1}$ is essentially sharp. 

 As alluded to above, the \eqref{P} property has strong consequences on the spectral theory of Schr\"odinger operators of the form $\Delta-V$ with singular potential $V\in L^p_{loc}$.
In the final Section~\ref{subsec:core} we will apply our main results to exhibit assumptions ensuring that the Schr\"odinger operator is essential self-adjoint (when $p=2$) or that $C^\infty_c(M)$ is an $L^p$-core for $\Delta-V$.

\section{Smooth approximation of distributional subsolutions}\label{sect:approx}

In the classical potential theory for the Euclidean Laplacian in $\rr^{n}$ (and therefore on any $2$-dimensional manifold in isothermal local coordinates) it is known that subharmonic distributions are the monotone limit of smooth subharmonic functions.
In particular, given a subharmonic function $u:\rr^n\to \rr$, $u\in L^1_{\text{loc}}\ton{\R^n}$, we can consider a $C^\infty_c(\R^n)$ mollifier $\phi_r(x)=r^{-n}\phi(x/r)$ and the one-parameter family of functions $u(x,r)=\phi_r \ast u$. Using standard estimates, one can show that for all $r>0$, $u(x,r)$ are smooth subharmonic functions that, as $r\to 0$, converge monotonically from above to $u(x)$ in the $L^1_{\text{loc}}$ sense.

We need to extend this property to the locally uniformly elliptic operator
 \[
 \CL = \Delta - \l(x)
 \]
on a Riemannian manifold, $\l(x)$ being a smooth function.

\begin{theorem}\label{th:smoothing}
	Let $(M,g)$ be an open manifold. For any $0\le \lambda \in C^\infty (M)$, the operator $\CL=\Delta-\lambda(x)$ has the  property of  local smooth monotonic approximation of $L^{1}_{loc}$-subsolutions, i.e the following holds.\\
For every $\Omega\Subset M$ relatively compact open domain with smooth boundary, there exists $\Omega'\Subset \Omega \Subset M$ such that ,  if $u \in L^{1}(\Omega)$ solves $\CL u \geq 0$ in $\Omega$, then there exists  a sequence $\{ u_{k} \} \subseteq C^{\infty}(\overline \Omega')$ satisfying the following properties:
\begin{enumerate}
 \item [a)] $u \leq u_{k+1}  \leq u_{k}$ for all $k \in \nn$;
 \item [b)] $u_{k}(x) \to u(x)$ as $k \to +\infty$ for a.e. $x \in \Omega'$;
 \item [c)] $\CL u_{k} \geq 0$ in $\Omega'$ for all $k \in \nn$;
 \item [d)] $\| u - u_{k} \|_{L^{1}(\Omega')} \to 0$ as $k\to +\infty$.
\end{enumerate}
\end{theorem}

We preliminary observe that, in order to prove the monotonic approximation, we can replace the Schr\"odinger operator $\CL$ by an operator of the form $\Delta_{\alpha} = \a^{-2}\div(\a^{2}\nabla \cdot)$ for a suitable smooth function $\a >0$. Indeed,  given a smooth neighborhood $\Omega \Subset M$ of $x_0$, let $\alpha\in C^\infty(M)$ be  a solution of the problem
\[
\begin{cases}
	\CL \alpha = 0,\\
    \alpha >0,
\end{cases}\text{on } \Omega.
\]	
Thanks to the maximum principle, such an $\alpha$ can be obtained for instance as a solution of the Dirichlet problem $\CL \alpha=0$ with positive constant boundary data on $\partial \Omega$. We use the following trick introduced by M.H. Protter and H.F. Weinberger in \cite{PW}.

\begin{remark}
The assumption $\lambda \ge 0$ in Theorem \ref{th:smoothing} can be avoided up to taking a small enough neighborhood of $x_0$. Indeed, suppose that $\lambda(x_0)<0$. Since the $M$ is asymptotically Euclidean in $x_0$, the constant in the Poincaré inequality can be made arbitrarily small on small domain. Namely, there exists a small enough neighborhood $U$ of $x_0$ so that 
	\[
 \int_{U} -\lambda \vp^2 \dv \le	-\frac{\lambda(x_0)}{2}\int_{U} \vp^2 \dv\le  \int_U |\nabla \vp|^2 \dv ,\quad \forall \vp\in C^\infty_c(U),\] i.e., the bottom of the spectrum $\lambda^U_1(-\CL)$ of $-\CL=-\Delta+\lambda$ on $U$ is non-negative. By the monotonicity of $\lambda^U_1(-\CL)$ with respect to the domain, we thus obtain that $\lambda^\Omega_1(-\CL)>0$ on some smaller domain $\Omega\Subset U$. Accordingly, there exists a strictly positive solution of $\CL \alpha =0$ on $\Omega$.
\end{remark}

\begin{lemma}\label{lemma_alpha} The function $w\in L^1_{loc}(M)$ is a distributional solution of $\CL w\ge 0$ on $\Omega$ if and only if $w/\alpha$ is a distributional solution of $\Delta_\alpha (w/\alpha):=\alpha^{-2}\dive(\alpha^2 \nabla(w/\alpha)) \ge 0$ on $\Omega$.
\end{lemma}
\begin{proof}
	Let $0\leq \vp \in C^\infty_c(\Omega)$. Since $\Delta\alpha = \lambda \alpha$, we have
	\begin{align*}
	\alpha \Delta_\alpha(\vp/\alpha) &= \alpha^{-1}\dive (\alpha^2 \nabla(\vp/\alpha))\\
	&= \alpha^{-1} \dive(\alpha\nabla \vp - \vp\nabla \alpha)\\
	&= \alpha^{-1} (\alpha\Delta \vp - \vp\Delta \alpha)\\
    &= \CL \vp.
	\end{align*}	
	Noticing that the operator
$\Delta_{\alpha}$ is symmetric with respect to the smooth weighted measure $\alpha^2\dv$, we get	
	\begin{align*}
		(\Delta_\alpha \frac{w}{\alpha}, \alpha\vp) = \int_{M} \frac{w}{\alpha}\Delta_{\alpha}\frac{\vp}{\alpha} \alpha^2\dv =
		\int_{M} \frac{w}{\alpha}\alpha^2\frac{\CL\vp}{\alpha} \dv =  
		\int_{M} w \CL\vp \dv = (\CL w,\vp). 
	\end{align*}
Since $0\le \alpha\vp \in C^\infty_c(\Omega)$, this concludes the proof of the lemma.
\end{proof}
	
According to this lemma, setting $v=\alpha^{-1}u$, we can infer the conclusion of the theorem from the equivalent statement for the operator $\Delta_\alpha$. Indeed, since $v\in L^1_{loc}(\Omega)$ solves $\Delta_\alpha v \geq 0$ distributionally in $\Omega$, if we prove that then there exists $\Omega'\Subset \Omega$ and a sequence $\{ v_{k} \} \subseteq C^{\infty}(\overline \Omega')$ satisfying the following properties:
\begin{enumerate}
	\item [a')] $v \leq v_{k+1}  \leq v_{k}$ for all $k \in \nn$;
	\item [b')] $v_{k}(x) \to v(x)$ as $k \to +\infty$ for a.e. $x \in \Omega'$;
	\item [c')] $\Delta_\alpha v_{k} \geq 0$ in $\Omega'$ for all $k \in \nn$.
    \item [d')] $\| v - v_{k} \|_{L^{1}(\Omega')} \to 0$ as $k\to +\infty$.
\end{enumerate}
then the sequence $u_{k} = \a v_{k}$ satisfies a), b), c) and d) as desired.\\

With this preparation,  Theorem \ref{th:smoothing} could be proved by exploiting the powerful machinery developed in the axiomatic potential theory. Indeed, 
according to \cite[Theorem 1]{Sj}, $v$ is a $\Delta_\alpha$-subharmonic function in the sense of Herv\'e,  \cite{He}. Hence, to conclude, we can apply a slightly modified version of \cite[Theorem 7.1]{BL}. Namely, one can verify that Theorem 7.1 in \cite{BL} works without any change if one uses in its proof the Green function of $\Omega$ with null boundary conditions on $\partial \Omega$. The existence of this Green function, in turn, can be deduced using different methods. For instance, from the PDE viewpoint,  we can appeal to the fact that the Dirichlet problem on smooth relatively compact domains is uniquely solvable; see the classical \cite{LSW}. A similar approach is used to obtain the monotonic approximation in \cite{BM}.

It is apparent that this kind of argument is quite involved. Therefore for the reader's convenience, we present a self-contained proof of Theorem \ref{th:smoothing} that does not rely on any abstract potential theory. Instead, it involves a rather different family of approximating functions with a convolution like flavor that, we feel, could be of independent interest.

\begin{proof}[Proof of Theorem \ref{th:smoothing}] Given the existence of local isothermal coordinates, and given that the Laplacian is conformally invariant in dimension $2$, the theorem follows easily in this case from the Euclidean case.

\vspace{5mm}

If $(M,g)$ has dimension $\geq 3$, by Lemma \ref{lemma_alpha} we can focus on the approximation problem for $\Delta_\alpha$, which is proportional to the Laplace-Beltrami operator relative to the metric $\tilde g_{ij} = \alpha ^{\frac 4 {n-2}} g_{ij}$. Thus, up to changing metric, we can simply study the approximation for the standard Laplace-Beltrami operator on the manifold $M$.

Let $\Omega$ be any relatively compact neighborhood of $x_0$, and let $\Omega'\Subset \Omega'' \Subset \Omega$, with $d(\Omega',\Omega''^C)\geq \epsilon>0$. Assume for simplicity that $\Omega$ has smooth boundary. Let $G_\Omega(x,y)=G(x,y)$ be the Green's function of the operator $\Delta$ on $\Omega$ with Dirichlet boundary conditions, where the signs are chosen so that $\Delta G = -\delta$. We recall some standard facts about the Green's function (see for example \cite[theorem 4.13]{AUGREEN}):
\begin{align}
 &G(x,y) \ \text{is smooth on } \ \Omega\times \Omega \setminus D\, , \ \text{where} \ D=\cur{(x,x)\, \ \ x\in \Omega} \ \text{is the diagonal}\\
 &G(x,y)>0\\
 &\label{G_est} G(x,y)\leq C d(x,y)^{2-n}\\
 &\label{dG_est} \abs{\nabla G(x,y)} \leq C d(x,y)^{1-n}\\
 &G(x,y)=G(y,x)\\
 &\label{G_Delta}\Delta_x G(x,y)= \Delta_y G(x,y)=-\delta_{x=y}
\end{align}

In the following, we will denote $G_x(y)=G(x,y)$ when $x$ is fixed. Thus we can also define the level sets $G_x^{-1}(t)= \cur{y \in M \ \ s.t. \ \ G_x(y)=t}\subset M$. Moreover, $\nabla_y G_x(y)=\nabla_y G(x,y)$ denotes the gradient of $G$ w.r.t. $y$. We will drop the subscript when there is no risk of confusion \footnote{ it is worth mentioning that while $G(x,y)$ is symmetric in $x$ and $y$, this does not mean that $\nabla_x G(x,y)=\pm \nabla_y G(x,y)$. Consider for example the symmetric function $f:\R^2\to \R$ given by $f(x,y)=xy$. However, since the usual Green's function in $\R^n$ is translation invariant, $G(x,y)=\abs{x-y}^{2-n}$, then in this case $\nabla_x G(x,y)=-\nabla_y G(x,y)$. This is not the case on a Riemannian manifold. }. 
It can be convenient sometimes to extend the definition of $G$ to $M\times M$ by setting it to be $0$ outside $\Omega\times \Omega$. This however turns $G$ into a non $C^1$ function over $\partial \Omega$, and \eqref{dG_est}, \eqref{G_Delta} hold only on $\Omega \times \Omega$.\\

Following a similar approach to the one in \cite{BL}, see also \cite{Ni} for a relevant representation formula for smooth functions, we can use the Green's function to construct explicitly the approximating sequence $v_k$. Let $\psi:\mathbb R \to [0,1]$ be a smooth function such that
\begin{gather}
\dot \psi \geq 0 \ \, \qquad \qquad \supp\ton{\dot \psi}\subseteq [-1;1]\ \, , \qquad \qquad \int_{\mathbb R} \dot \psi=1\, .
\end{gather}

Define the one parameter family of functions 
\begin{gather}
 v(x,r)= \int_M \dot \psi\ton{G(x,y)-r} u(y)\abs{\nabla_y G(x,y)}^2\ \dv(y)\, .
\end{gather}
For convenience, we will drop the integration symbol $\dv(y)$ when there is no risk of confusion.

Since $\supp\ton{\dot \psi}\subseteq[-1;1]$, and since $G(x,y)$ is smooth as long as $x\neq y \ \ \Longleftrightarrow G(x,y)\neq \infty$, the function $v(x,r)$ is smooth in $x$ for all $r<\infty$ fixed. Using the coarea formula, we can re-write $v(x,r)$ as 
\begin{gather}
 v(x,r)= \int_M \ \dot \psi\ton{G(x,y)-r} u(y)\abs{\nabla G(x,y)}^2=\\ \nonumber= \int_{-\infty}^\infty ds \ \dot \psi(s-r)\int_{G_x^{-1}(s)} u(y)\abs{\nabla G(x,y)}\, .
\end{gather}
Philosophically, $v(x,r)$ plays the role of a mollified $u(x)$ on the ball $\B {r^{2-n}}{x}$ as $r\to \infty$. 
 \vspace{5mm}
 
\textbf{Staying away from the boundary of $\Omega$}. A technical point needed for the proof of all the properties that we want to show is that we need to ``stay away from the boundary of $\Omega$''. What this means will become clear in the computations below, but for the moment let us mention that we can fix $r_0 \gg 0$ in such a way that for all $r\geq r_0$:
 \begin{gather}\label{eq_r_0}
  \forall x\in \Omega' \, ,  \ \forall y \in \Omega''^C \qquad \qquad \psi( G(x,y)-r)=0\, . 
 \end{gather}
This is possible because of the definition of $\psi$ and the upper bounds in \eqref{dG_est}.

Our final approximating sequence will be defined by
\begin{gather}
 v_k(x)= v(x,r_0+k)\, .
\end{gather}
 \vspace{5mm}
 
\textbf{Proof of monotonicity}.
Assuming that $u$ is smooth, and denoting $n$ the outward normal to the level sets $G_x^{-1}(t)$, we compute the derivative of $v(x,r)$ wrt the parameter $r$ by:
 \begin{gather}
  \frac \partial{\partial r} v(x,r)= \int_\R ds -\ddot \psi(s-r) \int_{G_x^{-1}(s)}\   u(y) \abs{\nabla G}= \\=\notag
  \int_\R ds -\ddot \psi(s-r) \int_{G_x=s}\   u(y) \ps{\nabla G}{-n}= \\ \notag
  = \int ds \ \ddot \psi(s-r) \int_{G_x\geq s} \ps{\nabla  u}{\nabla G} + \underbrace{\int ds \ \ddot \psi(s-r)}_{=0}  \underbrace{\int_{G_x\geq s} u(y)\Delta_y G_x}_{=u(x) \ \ \text{indep. of } s}  =\\[10pt]=\notag
  \int ds \ \ddot \psi(s-r) \int_s^\infty dt \int_{G_x=t} \ps{\nabla  u}{\frac{\nabla G}{\abs{\nabla G}}}=
 \end{gather}
 integrating by parts, since $\dot \psi$ has compact support the boundary term vanishes and we're left with:
 \begin{gather}
  \notag= +\int ds \ \dot \psi(s-r) \int_{G_x=s} \ps{\nabla u}{\frac{\nabla G}{\abs{\nabla G}}}= -\int ds \ \dot \psi(s-r) \int_{G_x=s} \ps{\nabla u}{n}=\\
  =\notag -\int ds \ \dot \psi(s-r) \int_{G_x\geq s} \Delta u=-\int_M \Delta u \int ds \ \dot \psi(s-r) \chi_{G_x\geq s}=\\
  \notag =\int_M \Delta u \int_{-\infty}^{G_x(y)} ds \ \dot \psi(s-r) =-\int_{M}\Delta u \cdot \psi(G_x(y)-r)\, .
 \end{gather}
 
We can conclude that
  \begin{gather}
  \frac{\partial}{\partial r} v(x,r)=-\int_{M}\Delta u \cdot \psi(G_x(y)-r)\, .
 \end{gather}
Now by the definition of $r_0$ in \eqref{eq_r_0}, if $r\geq r_0$ then for all $x\in \Omega'$ the support of $\psi(G_x(y)-r)$ is contained in $\Omega''$, and so $\psi(G_x(y)-r)$ is a non-negative smooth function of $y$ \footnote{the problem is that $G(x,y)$ is the Green's function of $\Omega$, so it is $C^0$ but not $C^1$ over $\partial \Omega$}. Since $\psi(s)=\int_{-\infty}^s \dot \psi (s)\geq 0$, we can use the distributional subharmonicity of $u$ to conclude that
  \begin{gather}
  \frac{\partial}{\partial r} v(x,r)=-\int_{M}\Delta u \cdot \psi(G_x(y)-r)=-\int_{M}u \cdot \Delta\qua{\psi(G_x(y)-r)}\leq 0
 \end{gather}
Since both the first and the third term in the last chain of equality pass to the limit as $u_j\to u$ in $L^1_{loc}$, this last inequality holds for all $u\in L^1$ that are distributionally subharmonic, smoothness of $u$ is not necessary here.
 \vspace{5mm}

\textbf{Proof of subharmonicity}. We proceed in a similar way for the proof of subharmonicity of $v(x,r)$ (here $r$ is fixed, and we study the subharmonicity in $x\in M$). Assuming as above that $u$ is smooth, we can compute
 \begin{gather}
  v(x,r)= \int_\R ds \ \dot \psi(s-r)  \int_{G_x=s}u\abs{\nabla G}= -\int_\R ds \ \dot \psi(s-r)  \int_{G_x=s}u\ \ps{\nabla G}{-n}=\\
  \notag =u(x)-\int_\R ds \ \dot \psi (s-r) \int_{G\geq s } \ps{\nabla u}{\nabla G}\, .
 \end{gather}
 Thus, rearranging the terms, we get
 \begin{gather} 
  v(x,r)-u(x)= -\int_\R ds \ \dot \psi (s-r) \int_{G\geq s } \ps{\nabla u}{\nabla G}=\\
  = \notag - \int_\R ds \ \dot \psi (s-r) \qua{\int_{G=s } G\ \ps{\nabla u}{n} - \int_{G_x\geq s} G \Delta u}=\\
  = \notag - \int_\R ds \ \dot \psi (s-r) \int_{G_x\geq s} (s-G) \Delta u=\\
  = \notag - \int_\Omega \Delta u  \int ds \ (s-G(x,y))\dot \psi (s-r) \chi_{G_x\geq s}
 \end{gather}
Let $\hat \psi(t)=\int_{-\infty}^t ds \psi(s)$ be the primitive of $\psi$. Notice that
\begin{gather}
 \int ds \ (s-G(x,y))\dot \psi (s-r) \chi_{G_x\geq s}=\\
 \notag =\int_{-\infty}^{G(x,y)} ds \  s\dot \psi (s-r) -G(x,y)\int_{-\infty}^{G(x,y)} ds \ \dot \psi (s-r) =\\
 \notag =G\psi(G-r)-\hat\psi(G-r)-G\psi(G-r)=-\hat \psi(G(x,y)-r) \, .
\end{gather}

Summing up, we have that, if $u$ is smooth,
\begin{gather}\label{eq_fr-u}
 v(x,r)-u(x)=\int_\Omega \hat \psi(G(x,y)-r) \Delta u(y)\, .
\end{gather}
We can compute the Laplacian of $v(x,r)$ by
\begin{gather}
  \Delta_x v(x,r)- \Delta_x u(x) = \int_M \ \Delta_x(\hat \psi(G(x,y)-r))\Delta_y u(y)\ \dv(y)\, .
\end{gather}
By direct computation:
\begin{gather}
 \Delta \qua{\hat \psi(G(x,y)-r)}= \nabla_i \qua{\psi (G(x,y)-r)\nabla_i G(x,y) } = \\
 \notag=\dot \psi(G(x,y)-r) \abs{\nabla G(x,y)}^2 +\psi(G(x,y)-r)\Delta G(x,y)=\\
 \notag=\dot \psi(G(x,y)-r) \abs{\nabla G(x,y)}^2 -\psi(G(x,y)-r)\delta_{x=y}=\dot \psi(G(x,y)-r) \abs{\nabla G(x,y)}^2 -\delta_{x=y}\, .
\end{gather}
Thus we obtain that 
\begin{gather}
 \Delta_x v(x,r)- \Delta_x u(x) = -\Delta_x u(x)+\int_\Omega \dot \psi(G(x,y)-r) \abs{\nabla_x G}^2 \Delta u(y)\dv(y)\, ,\\
 \notag \Delta_x v(x,r)=\int_\Omega \dot \psi(G(x,y)-r) \abs{\nabla_x G}^2 \Delta u\dv(y) =\int_\Omega u \ \Delta_y \qua{\dot \psi(G(x,y)-r) \abs{\nabla_x G}^2 }\dv(y)\, .
\end{gather}
The last equality makes sense also if $u$ is simply in $L^1_{\text{loc}}$ and not smooth. 

Arguing as in the proof of monotonicity, if $r\geq r_0$ and $x\in \Omega'$, $G(x,y)$ is smooth in this integral, and the distributional subharmonicity of $u$ allows us to conclude that
\begin{gather}
 \Delta_x v(x,r)\geq 0\, ,
\end{gather}
as desired.
 
 \vspace{5mm}
 
 \textbf{Proof of $L^1$ convergence}. In order to show that $v(x,r)\to v(x)$ in the $L^1(\Omega')$ sense as $r\to \infty$, we will use the equality \eqref{eq_fr-u} and the fact that, in the distributional sense, $\Delta u$ is a finite (non-negative) measure on the set $\Omega''$.

\medskip

Observe that if $u$ is a smooth subharmonic function, then by \eqref{eq_fr-u}:
 \begin{gather}
  v(x,r)-u(x)=\int_\Omega \hat \psi(G-r) \Delta u\, .
 \end{gather}
Thus the $L^1(\Omega')$ norm of $v(x,r)-u(x)$ is
\begin{gather}
 \int_{\Omega'}  \abs{v(x,r)-u(x)} \dv(x)=\int_{\Omega'}  v(x,r)-u(x)\ \dv(x)=\\
 \notag =\iint_{{\Omega'}\times {\Omega}} \hat \psi(G(x,y)-r) \Delta u(y)\ \dv(x)\dv(y)
 \notag =\int_{\Omega} \Delta u(y) \underbrace{\int_{\Omega'} \ \hat \psi(G(x,y)-r)\dv(x)}_{:=h(y,r)} \dv(y)\, .
\end{gather}
By the choice of $r_0$ in \eqref{eq_r_0}, the support of $h$ is contained in $\overline{\Omega''}$.

We claim that $h$ is a Lipschitz function and $\lim_{r\to \infty}\abs{h(y,r)}_{L^\infty(\Omega'')}=0$. Indeed, we observe that
\begin{gather}
 \nabla h = \int dx \ \psi(G(x,y)-r)\nabla G(x,y)\, ,
\end{gather}
and as long as $r\geq r_0$ and $x\in \Omega'$, $\psi(G(x,y)-r)\nabla G(x,y)$ is a smooth function of $y$ away from $x$ and supported in $\Omega''$ and the estimates in \eqref{dG_est} hold. Thus we obtain that $h$ is a Lipschitz function. 

Moreover, note that
\begin{gather}
 \hat \psi(G(x,y)-r)\leq \max\cur{G(x,y)-r+1;0}\leq \max \cur{C d(x,y)^{2-n}-r+1;0}
\end{gather}

Now, for a subharmonic function $u\in L^1$, we can use this fact to prove that the sequence $v(x,r)$ is an $L^1$ Cauchy sequence as $r\to \infty$. Indeed, for $R>r>0$,
\begin{gather}
 \int_{\Omega'}  \abs{v(x,r)-v(x,R)} \dv(x)=\int_{\Omega'}  v(x,r)-v(x,R)\ \dv(x)=\\
 \notag =\iint_{{\Omega'}\times {\Omega}} \qua{\hat \psi(G(x,y)-r)-\hat \psi(G(x,y)-R)} \Delta u(y)\ \dv(x)\dv(y)=\\
 \notag =\int_{\Omega} \Delta u(y) \underbrace{\int_{\Omega'} \ \qua{\hat \psi(G(x,y)-r)-\hat \psi(G(x,y)-R)} \dv(x)}_{:=h(y,r,R)} \dv(y)\, .
\end{gather}
$h(y,r,R)$ is a smooth function, so this equality makes sense distributionally. Moreover, $h(y,r,R)$ is a uniformly bounded sequence in $C^0$, that converges to $0$ as $r\to \infty$. Thus $v(x,r)$ is an $L^1$ Cauchy sequence in $r$. 

In order to prove that $v(x,r)\to u(x)$ as $r\to \infty$, we will prove that a convex combination of $v(x,r)$ for some $r\geq r_0$ converges to $u(x)$. This and the $L^1$ convergence of $v(x,r)$ prove the original convergence. 

In particular, for all $r_0$ fix some function $\alpha_{r_0}(r)$ such that
\begin{itemize}
 \item $\alpha_{r_0}\in C^\infty_c([0;\infty))$
 \item $\int \alpha_{r_0}(r) \ dr=1$
 \item $\forall r, \ 0\leq \alpha_{r_0}(r)\leq \frac 1 r$
 \item if $r\leq r_0$ or $r\geq 10 r_0$, $\alpha_{r_0}(r)=0$
\end{itemize}
Since $\int_{r_0}^{10 r_0} \frac {dr}{r} = \ln(10)>1$, this is possible. Consider the convex combinations
\begin{gather}
 w(x,r_0):=\int \alpha_{r_0}(r) v(x,r) dr = \int_0^\infty ds \int_{G(x,y)=s} u(y) \abs{\nabla G} \int dr \alpha_{r_0}(r)\dot \psi(s-r)=\\
 \notag = \int_{M} u(y) \abs{\nabla G}^2 \int dr \alpha_{r_0}(r)\dot \psi(G(x,y)-r)\dv(y)\, .
\end{gather}
Notice that by Cheng-Yau gradient estimates applied on a ball $\B {d(x,y)/3} x$, the positive harmonic function $G$ satisfies 
\begin{gather}
 \frac{\abs{\nabla G}}{G}\leq c (d(x,y)^{-1}+1)\, ,
\end{gather}
and since $G(x,y)\leq C d(x,y)^{2-n}$, for $\bar r\gg1$ we have
\begin{gather}
 \abs{w(x,\bar r)-u(x)}\leq \int_{M} \abs{u(y)-u(x)} \abs{\nabla G}^2 \int dr \alpha_{\bar r}(r)\dot \psi(G(x,y)-r)\dv(y)\leq \\
 \notag \leq \int_{M} \abs{u(y)-u(x)} d(x,y)^{-2} G^2(x,y) \frac{C}{G(x,y)-1} 1_{\cur{\bar r-1\leq G(x,y)\leq 10 \bar r+1}}\dv(y)\leq\\
 \notag \leq C\int_{A_{\bar r}(x)} \frac{1}{d(x,y)^n}\abs{u(y)-u(x)}\dv(y) \, ,
\end{gather}
where $A_{\bar r}(x)=\cur{(c^{-1}(10 \bar r+1))^{1/(2-n)}\leq d(x,y)\leq (C^{-1}(\bar r-1))^{1/(2-n)}}$. The last inequality is due to the local estimate $G(x,y)\ge c\,d(x,y)^{2-n}$ on $\Omega'$; see for instance \cite[Theorem 2.4]{MRS}. Let $x$ be a Lebesgue point for $u$, then $\abs{w(x,\bar r)-u(x)}\to 0$, and this proves a.e. convergence to $u(x)$ for the convex combination $w(x,\bar r)$, and thus also for the original $v(x,r)$.

\end{proof}

\vspace{5mm}

In the following, we will assume that either
$\l(x) \equiv 0$ or $\l(x) = \l$ is a positive constant. For each of these choices of $\l$, the existence of a smooth local monotone approximation has striking consequences in the regularity theory of  subharmonic distributions or on the validity of a variant of the traditional {\it Kato inequality}.

In the next two sections we are going to analyze separately these applications.

\section{Improved regularity of positive subharmonic distributions}\label{sect:reg}

By a {\it subharmonic distribution} on a domain $\Omega \subseteq M$ we mean a function $u \in L^{1}_{loc}(
\Omega)$ satisfying the inequality $\Delta u \geq 0$ in the sense of distributions. Note that, in this case, $\Delta u$ is a positive Radon measure.

Using the fact that a local monotone approximation exists, positive subharmonic distributions are necessarily in $W^{1,2}_{loc}$. Indeed an even stronger property can be proved, i.e., $u^{p/2}\in W^{1,2}_{loc}$ for all $p\in (1,\infty)$. This is the content of the next Theorem. To the best of our knowledge, in this generality the result is new also in the Euclidean setting.

\begin{theorem}\label{prop:reg positive subharm}
	Let $(M,g)$ be a Riemannian manifold. Let $u \geq 0$ be an $L^1_{loc}(M)$-subharmonic distribution. Then, $u\in L^\infty_{loc}$ and $u^{s/2}\in W^{1,2}_{loc}$ for any $s\in (1,\infty)$.\\[-10pt]
	
	Moreover, for all $\vp \in C^\infty_c(M)$ and $1<s\leq p$, we have the estimate
	\begin{align}\label{eq:caccioppoli-vp}
	 \frac{(s-1)^2}{s^2}\int_{\{\vp\ge1\}} |\nabla u^{s/2}|^2 &\le 
	 \int_M   u^{s}|\nabla \vp|^{2} 
	 \le \left(\int_{\operatorname{supp}(\nabla \phi)} u^p\right)^{s/p}\left(\int_M|\nabla \vp|^{\frac{2p}{p-s}}\right)^{(p-s)/p}\, ,
    \end{align}
\end{theorem}

\begin{proof}
Fix $1 < s\le p < +\infty$ and  let $\vp\in C^\infty_c(\Omega)$ to be a cut-off function such that $0\le \vp\le 1$ and $\vp\equiv 1$ on some bounded open set $\Omega_{1}\neq \emptyset$. Fix a bounded open set $\Omega$ such that $\supp(\vp)\Subset\Omega\Subset M$. According to Theorem \ref{th:smoothing} there is a smooth  approximation of $u$
\[
u_{k} \geq u_{k+1} \geq u \geq 0
\]
by (classical) solutions of $\Delta u_{k} \geq 0$ on $\Omega$. Note that, up to replacing $u_k$ with $u_k+1/k$, we can suppose that  $u_k>0$. 
Now, define $\psi = u_k^{s-1} \vp^{2}$.
	Since the $u_k$'s are smooth, strictly positive and subharmonic, we obtain
	\begin{align}\label{eq_pre_caccioppoli}
		0 &\geq -\int_M \Delta u_k \psi \\
		\notag& = \int_M g(\nabla u_k,\nabla \psi)\\
		\notag&= (s-1)\int_{M} \vp^{2}u_k^{s-2}|\nabla u_k |^{2} + 2\int_{M} \vp u_k^{s-1}g(\nabla u_k,\nabla \vp)\\
		\notag&\ge (s-1-\ve)\int_{M} \vp^{2}u_k^{s-2}|\nabla u_k |^{2} - \ve^{-1} \int_{M} u_k^{s}|\nabla \vp |^{2},
	\end{align}
	for any $\ve \in (0,s-1)$. Elaborating, we get the Caccioppoli inequality
	\begin{align}
		\ve(s-1-\ve)\int_{\Omega_{1}} u_k^{s-2}|\nabla u_k|^{2}
		&\leq	\ve(s-1-\ve)\int_{M} \vp^2 u_k^{s-2}|\nabla u_k|^{2}\\ \notag
		&\leq \int_{M} u_k^{s}|\nabla \vp|^{2}\\
		\notag&\leq \|\nabla \vp\|_\infty^2\int_{\Omega} u_s^{p}.
	\end{align}
	Choosing $\epsilon=(s-1)/2$ and using the fact that $u_k\ge u_{k+1}> 0$, this latter implies
	\begin{gather}\label{e:Caccioppoli intermediate}
	\int_{\Omega_{1}} |\nabla u_k^{s/2}|^{2}=\frac{s^2}{4}	\int_{\Omega_{1}} u_k^{s-2}|\nabla u_k|^{2}\leq	\frac{s^2}{(s-1)^2} \int_{\Omega} u_k^{s}|\nabla \vp|^2\leq	\frac{s^2\|\nabla \vp\|_\infty^2}{(s-1)^2} \int_{\Omega} u_1^{s},
	\end{gather}
	Noticing also that $\|u_k^{s/2}\|_{L^2(\Omega_{1})}\le \|u_1^{s/2}\|_{L^2(\Omega_{1})}$, we have thus obtained that the sequence $\{u_k^{s/2}\}$ is  bounded in $W^{1,2}(\Omega_{1})$, hence weakly converges in $W^{1,2}(\Omega_{1})$ (up to extract a subsequence) to some $v \in W^{1,2}(\Omega_{1})$. Since $u_k^{s/2}$ converges point-wise a.e. to $u^{s/2}$, it holds necessarily $v =u^{s/2}$ a.e. on $\Omega_{1}$.  In particular, $u^{s/2}\in W^{1,2}$ in a neighborhood in $\Omega_1$. Moreover, letting $k\to\infty$ (along the subsequence) in  \eqref{e:Caccioppoli intermediate}, we obtain
	\begin{gather}
		\int_{\Omega_{1}} |\nabla u^{s/2}|^{2}\leq\frac{s^2}{(s-1)^2} \int_{\Omega} u^{s}|\nabla \vp|^2\leq	\frac{s^2}{(s-1)^2}
		\left(\int_{\operatorname{supp}(\nabla \phi)} u^p\right)^{s/p}\left(\int_M|\nabla \vp|^{\frac{2p}{p-s}}\right)^{(p-s)/p},
	\end{gather}
where we used the H\"older inequality in the last step.	 \\[-10pt]

\end{proof}

\section{A variant of the Kato inequality}\label{sec:Kato}

The original inequality by T. Kato, \cite[Lemma A]{Ka}, states that if $u \in L^{1}_{loc}(M)$ satisfies $\Delta u \in L^{1}_{loc}(M)$ then
\[
\Delta | u | \geq  \sgn(u) \Delta u
\]
or, equivalently, if we set $u_{+} = \max(u,0) = (u + |u|)/2$, it holds
\[
\Delta u_{+} \geq  1_{\{ u>0\}} \Delta u.
\]
Note that, in these assumptions, $|\nabla u| \in L^{1}_{loc}(M)$ (see \cite[Lemma 1]{Ka}) and, therefore, $u \in W^{1,1}_{loc}(M)$.

Under the sole requirement that $u \in W^{1,1}_{loc}(M)$ is such that $\Delta u = \mu$ is a (signed) Radon measure, a precise form of the Kato inequality was proved by A. Ancona in \cite[Theorem 5.1]{An} elaborating on ideas contained in the paper \cite{Fu} by B. Fuglede.

In the special case where $M = \rr^{n}$ and $\Delta$ is the Euclidean Laplacian,  H. Brezis, \cite[Lemma A.1]{Br} and  \cite[Proposition 6.9]{Po}, observed that the local regularity of the function can be replaced by the condition that $u$ satisfies a differential inequality of the form $\Delta u \geq f$ in the sense of distributions, where $f \in L^{1}_{loc}$. The proof uses standard mollifiers to approximate $u$ by smooth solutions of the same inequality and, in particular, it works locally on $2$-dimensional Riemannian manifolds thanks to the existence of isothermal  local coordinates.  In the next result, we extend its validity to higher dimensional manifolds by using the existence of a  sequence of smooth subharmonic approximations.

\begin{proposition}[Brezis-Kato inequality]\label{prop-BK}
Let $(M,g)$ be a Riemannian manifold. If $u \in L^{1}_{loc}(M)$ satisfies $\Delta u \geq f$ in the sense of distributions, for some $f \in L^{1}_{loc}(M)$, then $u_{+} \in L^{1}_{loc}(M)$ is a distributional solution of  $\Delta u_{+} \geq 1_{\{u>0\}} f$.
\end{proposition}

We shall use the following approximation Lemma, see \cite[Lemma 2]{Ka}.

\begin{lemma}\label{lemma-approx}
 Let $u \in L^{1}_{loc}(M)$ satisfy $\Delta u \in L^{1}_{loc}(M)$. Then, for any fixed compact coordinate domain $\Omega \Subset M$, there exists a sequence of functions $u_{k} \in C^{\infty}(\Omega)$ such that
 \[
 u_{k} \overset{L^{1}(\Omega)} {\longrightarrow} u\quad \text{and}\quad \Delta u_{k} \overset{L^{1}(\Omega)}{\longrightarrow} \Delta u.
 \]
\end{lemma}

\begin{proof}[Proof (of Proposition \ref{prop-BK})]
 We fix a smooth coordinate domain $\Omega \Subset M$ where subharmonic distributions posses a monotone approximation by smooth subharmonic functions. We consider the Dirichlet problem
 \[
\begin{cases}
 \Delta g = f & \text{ in }\Omega \\
 g = 0 & \text{ on }\partial \Omega
\end{cases}
 \]
and we note that, since $f \in L^{1}(\Omega)$, then it has a unique solution $g\in W^{1,1}_{0}(\Omega)$; see \cite[Theorem 5.1]{LSW}. The new function
 \[
 w = u - g \in L^{1}(\Omega)
 \]
is a subharmonic distribution, namely, $\Delta w \geq 0$.
 
Let  $w_{k} \geq w_{k+1} \geq w$ be a monotonic approximation of $w$ by smooth solutions of $\Delta w_{k} \geq 0$ and define
 \[
 u_{k} = w_{k} + g .
 \]
 Note that, by construction,
 \begin{equation}\label{approx}
 i)\, u_{k} \searrow u \text{ a.e. in }\Omega, \quad ii)\, u_{k} \to u \text{ in }L^{1}(\Omega), \quad iii)\,  \Delta u_{k} \geq f \text{ in }\Omega.
 \end{equation}
Moreover, since the $w_{k}$ are smooth,
 \[
 u_{k} \in L^{1}(\Omega) \quad \text{and} \quad \Delta u_{k} \in L^{1}(\Omega).
 \]
According to Lemma \ref{lemma-approx}, for each fixed $k$, let $\{ u_{k}^{n}\}_{n \in \nn}$ be a sequence of smooth functions satisfying
\[
u_{k}^{n} \overset{L^{1}(\Omega)}{\longrightarrow} u_{k} \quad \text{and} \quad \Delta u_{k}^{n} \overset{L^{1}(\Omega)}{\longrightarrow} \Delta u_{k},\quad \text{as }n\to+\infty.
\]
Now, let $H:\mathbb R \to \mathbb R$ be a smooth function satisfying $H'(t),H''(t) \geq 0$. For any $n \ge 1$,
\begin{align*}
	\Delta (H (u_{k}^{n})) \ge H'(u_{k}^{n}) \Delta u_{k}^{n}.
\end{align*}
For every $0\le \vp\in C^\infty_c(\Omega)$,
\begin{equation}\label{eq:Hve}
\int_{\Omega} H(u_k^{n})\Delta \vp \dx = \int_{\Omega} \Delta (H(u_k^{n}))\vp \dx 
\ge \int_{\Omega} \Delta u_k^{n} H'(u_k^{n})\vp \dx.
\end{equation}
We apply this latter with $H(t)=H_\ve(t)=(t+\sqrt{t^2+\ve})/2$. First, we let $n \to \infty$. 
Using the dominated convergence theorem we get
\[
\left|\int_{\Omega} (H_\ve(u_k^{n})-H_\ve(u_{k})) \Delta \vp \dx\right| \leq \|H_\ve'\|_{L^\infty}\| \Delta\vp\|_{L^\infty}\int_{\Omega} |u_k^{n}-u_{k}|\dx \longrightarrow 0,
\]
and 
\begin{align*}
&\left|\int_{\Omega} \Delta u^{n}_k H_\ve'(u_k^{n})\vp \dx-\int_{\Omega} \Delta u_{k}H_\ve'(u_{k})\vp \dx\right|\\
\le  &\left|\int_{\Omega} (\Delta u_k^{n}- \Delta u_{k}) H_\ve'(u_k^{n})\vp \dx\right|+\left|\int_{\Omega} \Delta u_{k} (H_\ve'(u_k^{n})-H_\ve'(u_{k}))\vp \dx\right|\\
\le  &  \|H_\ve'\|_{L^\infty}\|\vp\|_{L^\infty}\int_{\Omega}  |\Delta u_{k}^{n}-\Delta u_{k}|  \dx+\int_{\Omega} \Delta u_{k}  |H_\ve'(u_k^{n})-H_\ve'(u_{k})|\, |\vp| \dx \longrightarrow 0.
\end{align*}
Therefore, \eqref{approx}, \eqref{eq:Hve} and the fact that $H_{\ve}',\vp \geq 0$, yield
\begin{equation}\label{approx1}
\int_{\Omega} H_\ve(u_{k}) \Delta \vp \dx   
\ge \int_{\Omega} H_\ve'(u_{k}) \Delta u_{k} \vp \dx \geq \int_{\Omega}H_{\ve}'(u_{k})f \vp \dx.
\end{equation}
Next we recall that $f \in L^{1}$ and, by \eqref{approx} ii), $u_{k} \to u$ in $L^{1}(\Omega)$. Thus, repeating the above estimates we can take the limit as $k \to +\infty$ in \eqref{approx1} and get
\[
\int_{\Omega} H_\ve(u) \Delta \vp \dx  \geq \int_{\Omega}H_{\ve}'(u)f \vp \dx.
\]
Finally, we note that $H_\ve(t) \to t_+$ uniformly on $\mathbb R$. Moreover the $H'_\ve(t)$ are uniformly bounded both in $\ve$ and in $t$, and converge pointwise a.e. as $\ve\to 0$ to the characteristic function $1_{(0,+\infty)}(t)$. Letting $\ve\to 0$ and applying again the dominated convergence theorem gives
\[
\int_{\Omega} u_+ \Delta\vp \dx   
\ge \int_{\Omega}  1_{\{ u>0 \}} f \vp \dx,
\]
i.e. 
\[
\Delta u \geq 1_{\{u>0\}} f\quad \text{distributionally in } \Omega.
\]
In order to conclude the proof, take a covering of $M$ by coordinate domains $\Omega_{j} \Subset M$ where the monotone approximation exists and consider a subordinated partition of unity $\{\eta_k\}$ such that $\eta_j\in C^\infty_c(\tilde \Omega_j)$ and $\sum_j\eta_j=1$. Given  $\psi\in C^\infty_c(M)$, one has
\[
\int_{M} u_+ \Delta\psi \dv   
=\sum_j \int_{M} u_+ \Delta(\eta_j\psi) \dv \ge \sum_j \int_{\tilde \Omega_j} 1_{\{ u>0\}} f  \eta_j \psi  \dv= \int_{M} 1_{\{u>0\}} f\psi \dv.
\]
\end{proof}

A direct application of Theorem \ref{prop:reg positive subharm} and Proposition \ref{prop-BK} gives the following

\begin{corollary}\label{coro:BrezisPonce}
Let $(M,g)$ be a Riemannian manifold. If $u \in L^{1}_{loc}(M)$ satisfies $\CL u = \Delta u - \l u \geq 0$, with $\l\ge 0$, then $u_{+} \in L^{1}_{loc}(M)$ is a non-negative subharmonic distribution, i.e., $\Delta u_{+} \geq 0$. In particular, $u_{+} \in L^{\infty}_{loc}(M)$ and, for any $p \in (1,+\infty)$, $u_{+}^{p/2} \in W^{1,2}_{loc}(M)$.
\end{corollary}

\begin{remark}
		It is also possible to prove Corollary \ref{coro:BrezisPonce} by adapting the proof in \cite[Lemma A.1]{Br} and  \cite[Proposition 6.9]{Po} to the Riemannian setting, up to  replacing the approximation by convolution used therein with the monotone approximation provided by Theorem \ref{th:smoothing}.
		However, this latter strategy seems not to work for general $f\in L^1_{loc}$ (i.e. not of the form $\lambda u$), as in Proposition \ref{prop-BK}. Although the  special case $f=\lambda u$ would be enough for what needed in this article, we decided to state and prove the general form of the Brezis-Kato inequality as we feel that it can be useful to study more general PDEs on manifolds.
\end{remark}

Once we have a Brezis-Kato inequality, the $L^{p}$-PP property  enters in the  realm of Liouville-type theorems for $L^{p}$-subharmonic distributions. This is explained in the next

\begin{lemma} \label{lemma-Lp-Sub}
	Let $(M,g)$ be any Riemannian manifold. Consider the following $L^p$-Liouville property for subharmonic distributions, with $p\in(1,\infty)$:
	\begin{equation}\label{Lp-Sub}\tag{$L-\displaystyle{Sub}_p$}
		\begin{cases}
			\Delta u  \geq 0  \text{ on }M\\
			u \geq 0 \text{ a.e. on }M\\
			u \in L^{p}(M),
		\end{cases}
		\quad \Longrightarrow \quad
		u \equiv const \text{ a.e. on } M.
	\end{equation}
	Then
	\[
	\eqref{Lp-Sub} \quad \Longrightarrow \eqref{P}.
	\]
\end{lemma}

\begin{proof}
	Suppose $u \in L^{p}(M)$ satisfies $\Delta u \leq u$. By Proposition \ref{prop-BK},
	\[
	\Delta (-u)_{+} \geq  (-u)_{+} \geq 0\text{ on }M,
	\]
	i.e. $(-u)_{+} \geq 0$ is a subharmonic distribution. Obviously, $(-u)_{+}\in L^{p}(M)$. It follows from \eqref{Lp-Sub} that $(-u)_{+} = 0$ a.e. on $M$. This means precisely that $u \geq 0$ a.e. on $M$, thus proving the validity of \eqref{P}.
\end{proof}

\section{Parabolicity, capacity and \texorpdfstring{\eqref{P}}{Pp} property}\label{sec:parabolic}

Recall that, given $\Omega\subset M$ a connected domain in
$M$ and $D\subset \Omega$ a compact set, for $1 \le p < \infty$, the $p$-capacity of $D$ in $\Omega$ is defined
by $\operatorname{Cap}_p(D, \Omega) := \inf \int_\Omega
|d\vp|^p$, where the infimum is among all Lipschitz functions compactly supported in $\Omega$ such that $\vp \ge 1$ on $D$. Moreover, $D$ is said to be $p$-polar if $\operatorname{Cap}_p(D, \Omega)=0$ for every $\Omega \Supset D$. Finally, $M$ is $p$-parabolic if there exists a compact set $D \subset M$ with non empty interior such that $\operatorname{Cap}_p(D, M)=0$.
In case $M$ has nonempty compact boundary $\partial M \not= \emptyset$, we call $M$ $p$-parabolic if its Riemannian double $\CD(M)$ is $p$-parabolic. This, in particular, applies to the ends  of a manifold with respect to a given smooth compact domain. These are understood in the wide sense of non-compact components obtained by removing the smooth compact set. Note also that any compact manifold with (or without) boundary is $p$-parabolic for every $1< q < +\infty$.

\begin{theorem}\label{th:caccio}
 Let $p\in(1,\infty)$. Let $M$ be an open connected (not necessarily complete) manifold with a finite number of ends $E_1,\dots, E_N$. If each end $E_j$ is $q_j$ parabolic for some  $\frac{2p}{p-1}<q_j\le\infty$, then \eqref{Lp-Sub} holds on $M$, so that in particular $M$ is $L^p$-PP.
\end{theorem}

\begin{proof}
Let $U$ be a compact set disconnecting $M$, so that $M=U\cup (\cup_{j=1}^N E_j)$ and let $u\in L^p(M)$ be a non-negative subharmonic function. Since $u\in L^{\infty}_{loc}$, we have that $\|u\|_{L^{\infty}(U)}<\infty$. Fix an end $E_J$ and define the new function 
	\[
	u_J:=\begin{cases}
		(u-\|u\|_{L^{\infty}(U)})_+,&\text{in }E_J,\\0,&\text{elsewhere}.
	\end{cases}
	\] 
	Note that also $u_J$ is non-negative, subharmonic and in $L^p(M)$. Let $\{M_k\}$ be an exhaustion of $M$, i.e. $M_k\Subset M_{k+1}$ are compact and $M=\cup_kM_k$, and assume wlog that $M_1\supset U$. By definition of $q_J$-parabolicity of the end $E_J$, we can find a sequence of smooth compactly supported cut-offs $\{\vp_k\}=\{\vp_{J,k}\}$ such that $\vp_k\equiv 1$ on $M_k\cup E_J^C$ and $\|\nabla\vp_k\|_{L^{q_J}(E_J)}\to 0$ as $k\to\infty$. Applying \eqref{eq:caccioppoli-vp} to $u_J$ with $\vp=\vp_k$ and $2p/(p-s)=q_J$, i.e. $s=\frac{q_J-2}{q_J}p$, we get that for all $k$:
    \begin{align}
	 \int_{\{\vp_k\ge1\}} |\nabla u_J^{s/2}|^2 &\le C\left(\int_{\operatorname{supp}(\nabla \phi_k)} u_J^p\right)^{s/p}\left(\int_M|\nabla \vp_k|^{q_J}\right)^{\frac 2{q_J}}\, ,
    \end{align}	
	Taking the limit as $k\to \infty,$ we deduce that $u_J$ is constant, hence null, on\label{key} $E_J$.	In particular $u$ is bounded on $E_J$, so that $u\in L^{\frac{pq_J}{q_J-2}}(E_J)$.
 	
     Now, by \eqref{eq:caccioppoli-vp} we get
 		\begin{align*}
 		\frac{(p-1)^2}{p^2}\int_{\{\vp\ge1\}} |\nabla u^{p/2}|^2 &\le\int_M   u^{p}|\nabla \vp|^{2} \\
 		&=\int_U   u^{p}|\nabla \vp|^{2} +\sum_{j=1}^N\int_{E_j}   u^{p}|\nabla \vp|^{2}
 	\end{align*}
 	for any $\vp\in C^{\infty}_c(M)$. Insert in this latter $\vp=\vp_k$, where $\{\vp_k\}$ is a family of cut-offs such that $\vp_k\equiv 1$ on $M_k$ and, for any $j=1,\dots,N$, $\|\nabla\vp_k\|_{L^{q_j}(E_j)}\to 0$ as $k\to\infty$. Thus,
 	\begin{align}\label{eq:caccioppoli-vp_1_ends} 
 		\int_{\{\vp_k\ge1\}} |\nabla u^{p/2}|^2 \le \frac{p^2}{(p-1)^2}\sum_{j=1}^N\left(\int_{E_j} u^{\frac{pq_j}{q_j-2}}\right)^{(q_j-2)/q_j}\left(\int_{E_j}|\nabla\vp_k|^{q_j}\right)^{2/q_j}
 	\end{align}
 goes to $0$ as $k\to\infty$. Hence $u$ is constant, and Lemma \ref{lemma-Lp-Sub} concludes the proof.
\end{proof}

Here we point out some immediate consequences of the previous theorem. 
First, we can interpret the $\infty$-parabolicity of the whole manifold as the geodesic completeness; see e.g. \cite{PS-Ensaios} and also \cite{AMP}. This implies the following  
\begin{corollary}\label{prop:Lp-Sub}
	Let $p\in(1,\infty)$. Let $(M,g)$ be a complete Riemannian manifold. Then \eqref{Lp-Sub}  holds, so that in particular $M$ is $L^p$-PP.
\end{corollary}
\begin{remark} \label{rem-subquadratic} As it is clear from the proof of Corollary \ref{prop:Lp-Sub}, on any given complete Riemannian manifold and for any $p\in(1,\infty)$, the Liouville property \eqref{Lp-Sub} holds in the stronger form:
	\begin{equation*}
		\begin{cases}
			\Delta u  \geq 0  \text{ on }M\\
			u \geq 0 \text{ a.e. on }M\\
			u \in L_{loc}^{p}(M)\text{ and }\|u\|^p_{L^p(B_{2k}(o)\setminus B_k(o))}=o(k^2),
		\end{cases}
		\quad \Longrightarrow \quad
		u \equiv c \text{ a.e. on } M.
	\end{equation*}
	Accordingly, also the Positivity Preserving property holds in this class of functions larger than $L^p(M)$.
\end{remark}

\begin{remark}\label{rmk-Liouville}
	The endpoint cases $p=1$ and $p=+\infty$ must be excluded.
	The failure of \eqref{Lp-Sub} for these values of $p$ is well known. Namely, the hyperbolic space supports infinitely many bounded (hence positive) harmonic functions whereas, on the opposite side, positive, non-constant, $L^{1}$-harmonic functions on complete Riemann surfaces with (finite volume and) super-quadratic curvature decay to $-\infty$ was constructed by P. Li and R. Schoen in \cite{LS}. While the existence of these functions tells us nothing about the failure of the $L^{p}$-PP property, counterexamples also to this latter have been provided in \cite{BM}.
\end{remark}

Similarly, restating Theorem \ref{th:caccio} in the simplest case, we have:
\begin{corollary}
	Let $p\in(1,\infty)$. Let $M$ be an open (possibly incomplete) $q$-parabolic manifold for some $\frac{2p}{p-1}<q\le\infty$. Then $M$ is $L^p$-PP.
\end{corollary}

The third corollary deals with manifolds of the form $(N\setminus K, h)$, where $(N,h)$ is a complete Riemannian manifold and $K$ is a compact set. Indeed, $p$-parabolicity of $N\setminus K$ is naturally related to the capacity of $K$ in $N$.
\begin{corollary}\label{coro:LpPP cap}
Let $p\in(1,\infty)$. Let $N$ be a complete Riemannian manifold and $K\subset N$ a compact set. Suppose that $K$ is $q$-polar for some $q>\frac{2p}{p-1}$. Then $M = N\setminus K$ is $L^p$-PP.
	\end{corollary}

\begin{proof}
Take a smooth relatively compact domain $U \Subset N$, with $K \subset U$. Note that $\bar U$ is $q$-parabolic as a compact manifold with smooth boundary. Removing interior $q$-polar sets does not affect the $q$-parabolicity of the space. Therefore, the end  $E_{0}= \bar D \setminus K$ of the open Riemannian manifold $M = N \setminus K$ is still $q$-parabolic. On the other hand, since $N$ is complete, $N \setminus D$ has a finite number of unbounded connected components $E_{1},\cdots,E_{k}$. Each of them is a complete, hence $\infty$-parabolic, end of $M$. A direct application of Theorem \ref{th:caccio} yields that $M$ is $L^{p}$-PP.
\end{proof}

	Suppose now that $\dim_{\mathcal H}(K)<q<n-\frac{2p}{p-1}$. Hence $\mathcal H^q(K)=0$. By standard potential theory this implies that $K$ is $(n-q)$-polar; see for instance \cite{HKM} or \cite[Theorem 3.5]{Tr}.
Accordingly, we have the following straightforward consequence of Corollary \ref{coro:LpPP cap}.
\begin{corollary}\label{coro:LpPP Haus}
	Let $p\in(1,\infty)$, let $N$ be an $n$-dimensional complete Riemannian manifold and let $K$ be compact set of $N$ such that $\dim_{\mathcal H}(K)<n-\frac{2p}{p-1}$. Then $N\setminus K$ is $L^p$-PP.
\end{corollary}

A natural question arising from Corollary \ref{coro:LpPP Haus} is what happens in the threshold case $\frac{2p}{p-1}=q$ and below the threshold. In this case Hausdorff co-dimension $\frac{2p}{p-1}$ is not enough to preserve the \eqref{P} property, but uniform Minkowski control is. We start by proving the latter statement, and give an example of failure for Hausdorff co-dimension in example \ref{ex_not_good}.
\begin{proposition}\label{prop_Mink}
 Let $(N,h)$ be a complete Riemannian manifold, and let $E$ satisfy a uniform Minkowski-type estimate of the form
 \begin{gather}
  \vol\ton{\B r E} \leq C r^{\frac{2p}{p-1}}\, \qquad \text{for some } \ p>1\, .
 \end{gather}
The $N\setminus \overline E$ is an open manifold that satisfies the \eqref{Lp-Sub} property, and hence also the \eqref{P} property.
\end{proposition}
\begin{remark} 	In the assumptions of Proposition \ref{prop:LpPPMink} we have in particular that all $L^2$ harmonic functions on $N\setminus \overline E$ are necessarily constant. In the very special case where $E$ is a point, this result has been recently proved in \cite[Theorem 8]{HMW}.
\end{remark}

\begin{remark}
For several significant examples of sets with non-integer Hausdorff dimension the Minkowski content is controlled at the right dimensional scale (this is the case for instance for auto-similar fractal sets). Accordingly, in all these cases Proposition \ref{prop_Mink} gives also a simpler and more direct proof of Corollary \ref{coro:LpPP Haus}. Moreover, a recursive application of (the proof of) Proposition \ref{prop_Mink} permits also to deal with sets $K$ whose Minkowski content is not finite at the right dimension, yet suitably controlled. For instance, this occurs whenever 
		$K=\cup_{j=0}^N K_j$ can be decomposed as the union of a finite increasing family of compact sets $K_j\subset K_{j+1}$ such that $K_0$ and $K_{j+1}\setminus K_j$, $j=0,\dots,N-1$, have finite $(n-\frac{2p}{p-1})$-dimensional Minkowski content, as in the toy example $K=\{0\}\cup \{1/j\}_{j\ge 1}\subset \rr\subset\rr^n$. However, there exist compact sets for which the local Minkowski dimension is larger than the Hausdorff dimension around any point of $K$, so that they satisfy the $L^p$-PP property, but this latter can not be deduced though the technique introduced in Proposition \ref{prop_Mink}. Examples presenting this feature are auto-affine sets, \cite{Be,McM,Ba}, i.e., roughly speaking auto-similar type fractal sets of $\rr^n$ whose auto-similarity factor changes according to the direction.
\end{remark}

The proof of Proposition \ref{prop_Mink} is a direct consequence of the following possibly standard removable singularity lemma for $L^p$ functions.
\begin{lemma}\label{lem:laplacian extension}
 
 Let $(N,h)$ be a complete Riemannian manifold, and let $E$ satisfy a uniform Minkowski-type estimate of the form
 \begin{gather}
  \vol\ton{\B r E} \leq \C r^{\frac{2p}{p-1}}\, \qquad \text{for some } \ p>1 \ \ \text{and all }\ r\in (0;1] \, .
 \end{gather}
If $u\in L^p (N\setminus E)=L^p(N)$ and the distributional Laplacian of $u$ on $N\setminus \overline E$ is non-negative, then the distributional Laplacian of $u$ on $N$ is non-negative.
\end{lemma}
\begin{proof}
 Notice that the uniform volume estimates imply that $E$ is a bounded set, and the estimates are stable under closure, meaning that $\vol\ton{\B r E} =\vol\ton{\B r {\overline E}}$. 

 Let $\psi_k$ be a sequence of cutoff functions with $\supp\ton{\psi_k}\subseteq \B{2k^{-1}}{E}$ and
 \begin{itemize}
  \item $\psi_k=1$ on $\B{k^{-1}}E$
  \item $\abs{\nabla \psi_k}_\infty\leq ck$
  \item $\abs{\nabla^2 \psi_k}_\infty \leq ck^2$. \end{itemize}
For instance, the $\psi_{k}$'s can be obtained by smoothing out the (Lip) distance function from $E$.  Let also $\vp\in C^\infty_c(N)$ be any non-negative test function. We have the distributional identity
 \begin{gather}
  \int_N \vp \Delta u =  \int_N \vp (1-\psi_k)\Delta u + \int_N \vp \psi_k \Delta u\, .
 \end{gather}
 Since $\vp (1-\psi_k)\in C^\infty_c(N\setminus \overline{E})$ and it is non-negative, by hypothesis $\int_N \vp (1-\psi_k)\Delta u\geq 0$. Moreover, we can estimate (here $A_k= \B {2k^{-1}}{E}\setminus \B {k^{-1}}{E}$):
 \begin{gather}
  \abs{\int_N \vp \psi_k \Delta u}\leq \int u \qua{\abs{\psi_k} \abs{\Delta \vp}_\infty +2\abs{\nabla \psi_k} \abs{\nabla \vp}_{\infty} + \abs{\Delta \psi_k}\abs{\vp}_\infty }\leq\\
  \notag \leq c \int u \qua{\abs{\psi_k} +\abs{\nabla \psi_k} + \abs{\Delta \psi_k}}\leq c\int_{\B {2k^{-1}}{E}} u +c \int_{A_k} k u + k^2 u \leq\\
  \notag \leq c\int_{\B {2k^{-1}}{E}} u +c \ton{\int_{A_k} u^p}^{\frac{1}{p}}\underbrace{\ton{k^{2q} \vol\ton{A_k}}^{\frac{1}{q}}}_{\leq \C}\leq c\qua{\int_{\B {2k^{-1}}{E}} u + \C \ton{\int_{A_k} u^p}^{\frac{1}{p}}}
 \end{gather}
where $q$ is the H\"older exponent $q=\frac{p}{p-1}$. Since $u\in L^p$, $\abs{\int_N \vp \psi_k \Delta u}\to 0$ as $k\to \infty$, and this concludes the proof.
\end{proof}

\begin{example}\label{ex_not_good}	We note that a control on the Hausdorff dimension in the threshold case is not enough. Namely, there exists a compact $K\subset \rr^4$ such that $\dim_{\mathcal H}(K)=0$ but $\rr^4\setminus K$ is not $L^p$-PP. 
	
	Such a $K$ can be constructed as a generalized Cantor set. Let $\cur{\alpha_j}_{j=1}^\infty$ be a sequence bounded between $0$ and $1$, and consider the generalized Cantor set $C_\alpha$ constructed by starting with $K_0=[0,1]$ and removing open middle $\alpha_1$-th of $K_0$ to produce $K_1$ and so on.

	Define $C_\alpha=\cap_{j=0}^\infty K_j$ and $K=C_\alpha\times\{0\}_{\rr^3}\subset \rr^4$. We choose $\alpha_j=\frac{10^{j}-2}{10^j}$. \footnote{The resulting set is a perfect set formed by all the numbers $x$ in $[0,1]$ such that in the decimal development of $x$ the digits between the $2^{j}$-th and the $(2^{j+1}-1)$-th are either all $0$ or all $9$.} Clearly, the $\dim_{\mathcal H}(C_\alpha)=0$ as it is smaller than $\log_{10^j}2=j^{-1}\log_{10}2$ for every $j\in\mathbb N$. Now, let $\mu$ be a Borel measure supported on $K$ such that $\mu(\rr^4)=\mu(K)=1$ and $\mu(I\times\{0\}_{\rr^3})=2^{-j}$ for every connected component $I$ of $K_j$. The measure $\mu$ can be constructed as follow. Let $\tilde{\mathcal H}^{\log_3^2}$ be the normalized $\log_32$-dimensional Hausdorff measure restricted to the standard mid-third Cantor set $C_{1/3}$, so that $\tilde{\mathcal H}^{\log_3^2}(C_{1/3})=1$. Namely, $C_{1/3}=C_{\alpha'}$ with $\alpha'_j=1/3$ for all $j$, and we name $K'_j$ the iterative steps in the construction of $C_{1/3}$. Consider a continuous increasing bijective function $f:[0,1]\to[0,1]$ constructed recursively as follows: for every $j\in\mathbb N$, $f$ maps $K'_{j+1}\setminus K'_j$ to $K_{j+1}\setminus K_j$. Define $\mu|_{C_\alpha}$ as the push-forward measure of $\tilde{\mathcal H}^{\log_3^2}$ through $f$ and extend it as zero to obtain a measure $\mu$ on the whole $\rr^4$. 

	An easy computation shows that $\diam(I)=\prod_{k=1}^{j}10^{-k}=10^{-\frac{j^2+j}{2}}$ for every connected component $I$ of $K_j$, while any two such components have distance at least $8$ times larger. Thus, for every $x\in\rr^4$, one has $\mu(B_\epsilon(x))\le 2^{-j}$ as soon as $\epsilon\le 4\cdot 10^{-\frac{j^2+j}{2}}$, i.e.
	$2\log_{10}(1/\epsilon)\ge  j^2+j-2\log_{10}(4)\ge j^2$, so that
	\[
	\mu(B_\epsilon(x))\le 2^{-\sqrt{\log_{10}(1/\epsilon^2)}}\le (\log^2(1/\epsilon))^{-2}
	\]
	for $\epsilon$ small enough. Applying a modified version of Lemma \ref{lem:valto-adam} (see Remark \ref{rmk:log growth}) we deduce that $\rr^4\setminus K$ is not $L^2$-PP. 
	
	Note that, in this example, $\mathcal H^0(K)=\infty$. We do not know if there exists examples of compact sets of $\rr^n$ of finite $(n-2p/(p-1))$-dimensional Hausdorff measure which are not $L^p$-PP.
\end{example}
 
In conclusion of this section we would like to point out how the proof of Lemma \ref{lem:laplacian extension} can be adapted to get straightforwardly an extension of \cite[Theorem 5]{GP} to the case where a set of small size is removed.

\begin{proposition}
\label{prop:density}
	Let $(N,h)$ be a complete Riemannian manifold. Fix $p>1$ and let $E$ satisfy a uniform Minkowski-type estimate of the form
	\begin{gather}
		\vol\ton{\B r E} \leq \C r^{q}\, \qquad \text{for some } \ q>p \ \ \text{and all }\ r\in (0;1] \, .
	\end{gather}
 Then the space $C^\infty_c(N\setminus E)$ is dense in the space 
	\[
	\widetilde W^{2,p}(N\setminus E):=\{u\in L^p(N\setminus E)\ :\ \Delta u\in L^p(N\setminus E)\}
	\]
	with respect to its canonical norm $\|f\|_{\widetilde W^{2,p}(N\setminus E)}^p=\|f\|^p_{L^{p}(N\setminus E)}+\|\Delta f\|^p_{L^{p}(N\setminus E)}$.
\end{proposition}
\begin{proof}
	Let $f\in \widetilde W^{2,p}(N\setminus E)$. Then $f	\in L^p(N)$, and by the proof of Lemma \ref{lem:laplacian extension} it is clear that the distributional Laplacian of $f$ on $N$, denoted again by $\Delta f$ is in $L^p(N)$.
	By a result due to O. Milatovic, there exists a sequence of functions $\eta_j$ in $C^\infty_c(N)$ which converges to $f$ in $\widetilde{W}^{2,p}(N)$ as $j\to\infty$; see \cite[Theorem 5]{GP}.
Fix $j$ and let $\psi_k$ be the cut-off functions introduced in the proof of Lemma \ref{lem:laplacian extension}. To conclude the proof, it suffices to show that $\psi_k\eta_j$ converges to $\eta_j$ in $\widetilde{W}^{2,p}(N\setminus E)$ as $k\to\infty$. To this end, compute 
	\begin{align*}
		\|\eta_j-\psi_k \eta_j\|_{\widetilde{W}^{2,p}(N\setminus E)}&= \|\eta_j(1-\psi_k)\|_{L^p(N\setminus E)}+\|(1-\psi_k)\Delta \eta_j\|_{L^p(N\setminus E)} \\
		&+ \| \eta_j\Delta\psi_k\|_{L^p(N\setminus E)} + 2\| |\nabla \eta_j||\nabla\psi_k|\|_{L^p(N\setminus E)}\nonumber
		\end{align*}
	The first two terms on the RHS vanish by the dominated convergence theorem. Moreover,
\begin{align*}
	&\| \eta_j\Delta\psi_k\|_{L^p(N\setminus E)}\le \|\eta_j\|_{L^\infty}\| \Delta\psi_k\|_{L^p(N\setminus E)}\le c \|\eta_j\|_{L^\infty} k^{2-2q/p}\to 0, \\&\| |\nabla \eta_j||\nabla\psi_k|\|_{L^p(N\setminus E)}\le \||\nabla \eta_j|\|_{L^\infty}\| |\nabla \psi_k|\|_{L^p(N\setminus E)}\le c \||\nabla \eta_j|\|_{L^\infty} k^{1-2q/p}\to 0, 
\end{align*}	
as $k\to\infty$.

\end{proof}

\section{Removing sets of large size}\label{sec_Kbig}
As one might expect, Corollary \ref{coro:LpPP Haus} doesn't hold if the set $K$ is too big. Counterexamples can be easily built by looking at solutions of $\Delta f= \mu$, where $\mu$ is the right dimensional Hausdorff measure restricted to $K$.

\begin{example}
 As a model example, consider $\R^n$ with $n\geq 3$ and a subset $K$ of the form $K= P \cap \overline{\B 1 0}$, where $P$ is a $k$-dimensional plane passing through the origin. If we focus on the manifold $\R^n\setminus K$, by Proposition \ref{prop_Mink} the \eqref{P} and \eqref{Lp-Sub} properties hold as long as $k\leq n-\frac{2p}{p-1}$.
 
 In the other cases, we can consider the measure $\mu_K= \CH^k|_K$ and consider the fundamental solutions of
 \begin{gather}
  \Delta u_1 = -\mu_K\, , \qquad \qquad -\Delta u_2 + u_2 = \mu_K\, .
 \end{gather}
These solutions $u_1$ and $u_2$ can be easily written in terms of the Green's function $\abs{x-y}^{2-n}$ and the Bessel potential $J_2(x-y)$:
\begin{gather}
 u_1(x)=\int_{\R^n} d\mu_K(y) \ \abs{x-y}^{2-n} \, , \qquad \qquad u_2(x)=\int_{\R^n} d\mu_K(y) \ J_2(x-y)
\end{gather}

By the estimates in \cite{adams1} (see also \cite[Lemma 10.12]{Po}), the functions $u_1$ and $u_2$ belong to $L^p(\R^n\setminus K) =L^p(\R^n)$, and they are clearly solutions of
\begin{gather}
 \Delta u_1 =0 \qquad \text{on} \ \ \R^n\setminus K\\
 \Delta u_2 -u_2=0 \qquad \text{on} \ \ \R^n\setminus K
\end{gather}
However, $u_1$ is not a constant, and $u_2$ is not non-negative (for example $u_2(x)\to -\infty$ as $x\to 0$). Thus $u_1$ is a counterexample to property \eqref{Lp-Sub} on $\R^n \setminus K$, and similarly $u_2$ is a counterexample to property \eqref{P} on $\R^n\setminus K$.

\end{example}

In general, we have the following.
\begin{proposition}\label{prop_no_good}
 Let $(N,h)$ be a complete Riemannian manifold of dimension $n\geq 2$, and $\emptyset\neq K\Subset N$ have Hausdorff dimension $k$. If $k>n-\frac{2p}{p-1}$ for some $p\in(1,\infty)$, then the open manifold $(N\setminus K,h)$ does not enjoy either the \eqref{Lp-Sub} or the \eqref{P} property. \end{proposition}
\begin{proof} 
		The proof of this proposition is fairly standard, and follows the ideas laid out in the previous example.
		
First, we remark that the result is trivial if 
$n=2$, as in this case for a fixed $y_0\in K$ the Bessel potential $J_2(x,y_0)$ is in $L^p$ and thus gives a counterexample to the \eqref{P} property. Similarly, the Green's function of a $2$ dimensional manifold is in $L^p$ for all $p\in (1,\infty)$, and can be used to produce a counterexample to the \eqref{Lp-Sub} property.

Moreover, if $k\ge n-2$, then we can replace $K$ with a subset $K'\subset K$ of Hausdorff dimension $k'\in (n-\frac{2p}{p-1},n-2)$. If $N\setminus K'$ does not enjoy the $L^p$-PP property, than \textit{a fortiori} also $N\setminus K$ does not enjoy it. 
Hence, in the following 
we can thus assume that
\begin{gather}\label{eq_kPDCD1K}
n\geq 3 \qquad \text{and}\qquad  n-2>k>n-\frac{2p}{p-1}\, .
\end{gather}

 Since all the ideas involved in the proof are local, it is convenient to assume that $K\Subset \Omega'\Subset \Omega$, where $\Omega$ is a relatively compact coordinate neighborhood and has smooth boundary. Thus we have at our disposal the fundamental solutions of $\Delta $ and $\Delta -1$ that are the Green's function $G(x,y)$ (with Dirichlet boundary conditions) on $\Omega$ and the Bessel potential $J_2(x,y)$ (with Dirichlet boundary conditions) on $\Omega$. For convenience, we choose the signs such that $G(x,y)\sim +c d(x,y)^{2-n}$, and similarly $J_2(x,y)\sim + c d(x,y)^{2-n}$, i.e.
\begin{gather}
 \Delta G(x,y)=-\delta_{x-y}\, , \qquad \Delta J_2(x,y)-J_2(x,y)=-\delta_{x=y}\, .
\end{gather}

 Recall that the Bessel potential can be defined for example with the Heat Kernel (see \cite[eq 4.2]{strichartz}) by:
\begin{gather}
 J_2(x,y)= \int_0^\infty dt \ e^{-t} H_t(x,y)\, .
\end{gather}
We recall the comparison with the Green's function:
\begin{gather}
 G(x,y) = \int_0^\infty dt H_t(x,y)\, .
\end{gather}
$J_{2}$ converges absolutely for almost all $x,y$ to a positive function, symmetric in $x,y$, and it is easy to see that, similarly to the Green's function, we have for all $x,y\in \Omega'$
\begin{gather}
 cd(x,y)^{2-n}\leq J_2(x,y)\leq C d(x,y)^{2-n}\, .
\end{gather}
Notice that the second inequality is valid also globally on $N$, see e.g. \cite[theorem 7.1]{LSW}.
\vspace{5mm}

\textbf{Measure estimates.} If $\operatorname{dim}(K)>k$, then the $k$-dimensional Hausdorff measure $\CH^k$ of $K$ is infinity. We want to show that there is a subset $S\subset K$ and a (potentially big) constant $C>1$ such that 
\begin{gather}\label{eq_PDMTDCB}
 \forall \, r ,\, x\, : \qquad \qquad \CH^k(\B r x \cap S) \leq Cr^k\, ,\\
\notag \exists x \ \ s.t. \ \ \limsup_{r\to 0} r^{-k}\CH^k(\B r x \cap S) \geq C^{-1}
\end{gather}

By the standard \cite[theorem 8.19]{mattilone} applied to $K$, there exists a compact subset $S_1\subset K$ with positive and finite $k$-dimensional Hausdorff measure. If we consider the upper density of the measure $\CH^k|_{S_1}$, i.e. the limit
\begin{gather}
 \Theta^{*,k}(x)=\limsup_{r\to 0} \frac{\CH^k(S_1\cap \B r x)}{\omega_k r^k}\, ,
\end{gather}
we have by \cite[theorem 6.2]{mattilone} that for $\CH^k$-almost all $x\in S_1$ 
\begin{gather}\label{eq_dens}
 2^{-k}\leq \Theta^{*,k}(x)\leq 1
\end{gather}
and $\Theta^{*,k}(x)=0$ for $\CH^k$-almost all $x\not\in S_1$.

Let $S_2\subseteq S_1$ be the subset where \eqref{eq_dens} holds, and let $r_u:S_2\to (0,\infty)$ be defined by
\begin{gather}
 r_u(x) = \sup \cur{r>0 \ \ s.t. \ \ \forall s\leq r: \ \ \CH^k(S\cap \B s x) \leq 7\omega_k s^k}\, .
\end{gather}
Since $r_u(x)>0$ for all $x\in S_2$, then the measurable subsets $\tilde S_i\subset S_2$ defined by
\begin{gather}
 \tilde S_i = \cur{x\in S_2 \ \ s.t. \ \ r_u(x)>i^{-1}}
\end{gather}
constitute a monotone sequence converging to $S_2$, and thus there exists some $\hat i$ such that 
\begin{gather}
 \CH^k(\tilde S_{\hat i})>\CH^k(S_1)/2>0\, .
\end{gather}
Now if we consider the subset $S=\tilde S_{\hat i}$, we have that for all $x\in \B {\hat i^{-1}}S$:
\begin{gather}
 \CH^k(\B r x \cap S) \leq \begin{cases}
                            7\omega_k (2r)^k & \text{if } \ r\leq \hat i^{-1}/2\, ,\\
                            \CH^k(S_1)<\infty& \text{if } \ r\geq \hat i^{-1}/2\, .
                           \end{cases}
\end{gather}
Thus we can conclude that there exists a constant $C$ (depending on $\hat i$) such that $\CH^k(\B r x \cap S) \leq Cr^k$ for all $x,r$. Moreover, since $\CH^k(S)>0$, by \cite[theorem 6.2]{mattilone}, $\Theta^{*,k}(S,x)>2^{-k}$ for $\CH^k$-almost all $x\in S$, and thus there exists a point $\bar x\in S$ such that 
\begin{gather}
 \limsup_{r\to 0} r^{-k}\CH^k(\B r {\bar x} \cap S) \geq 2^{-k}\, .
\end{gather}

Thus we have proved all the desired properties in \eqref{eq_PDMTDCB}. For convenience, from now on we will use the notation
\begin{gather}
 \mu_K(E):= \CH^k(S\cap E)\, 
\end{gather}
to indicate a measure satisfying \eqref{eq_PDMTDCB}.

\medskip

\textbf{Counterexamples to the \ref{Lp-Sub} and \ref{P} properties.} We will prove that the fundamental solutions of
 \begin{gather}\label{e:fundamental}
  \Delta u_1 = -\mu_K\, , \qquad \qquad -\Delta u_2 + u_2 = -\mu_K\, .
 \end{gather}
with Dirichlet boundary conditions satisfy $u_1,u_2\in L^p$, $u_1(x)$ is positive but not constant and $u_2(x)$ is not bounded from below.

\medskip

\textbf{Lower bounds.} First of all, we notice that if $x$ is a point of density for $\mu_K$, i.e. if $\limsup_{r\to 0} r^{-k} \mu(\B r x )\geq c>0$, then 
\begin{gather}
 \limsup_{y\to x} u_1(y)=\limsup_{y\to x} \int G(z,y)d\mu_K(z)\geq \limsup_{y\to x} \int_{\B r x} G(z,y)d\mu_K(z)\, .
\end{gather}
Given that there are infinitely many radia $r_i\to 0$ such that $\mu_K(\B {r_i}{x})\geq c r_i^{k}$, we have that if $d(x,y)\sim r_i$:
\begin{gather}
 u_1(x) \geq c r_i^{2-n} r_i^{k} \to +\infty\, ,
\end{gather}
where we used \eqref{eq_kPDCD1K}. In a similar way, we obtain that around a density point $x$ we have
\begin{gather}
 \limsup_{y\to x} u_2(x) =-\infty\, .
\end{gather}

The most complicated part of the proof are the $L^p$ estimates, but, up to trivial adaptation to the Riemannian setting, these are contained in \cite{adams1} (see also \cite[Lemma 10.12]{Po}). For the reader's convenience, we report a proof in the following lemma, which will conclude the proof of this proposition.
\end{proof}

 \begin{lemma}\label{lem:valto-adam}
  Let $0<s< n-2$ and $\mu$ be a positive measure with support in $\B 1 P$ for some $P$, and suppose that 
  \begin{gather}\label{e:morrey growth}
   \mu(\B r x)\leq C r^s\, , \qquad \forall x,r\, .
  \end{gather}
  Suppose also that $\operatorname{Vol}(\B r x)\leq C r^n$ for all $x\in \Omega'$ and $r\leq 2$. Then
  \begin{gather}
   \int G(x,y) d\mu(y)\in L^p\, , \qquad p\in \ton{\frac{n}{n-2},\frac{n-s}{n-2-s}}\, ,\\
   \int J_2(x,y)d\mu(y)\in L^p\, , \qquad p\in \ton{\frac{n}{n-2},\frac{n-s}{n-2-s}}\, .
  \end{gather}
 \end{lemma}
\begin{proof}
 We consider the Green's function case, since the Bessel potential estimates are completely analogous. We have that 
 \begin{gather}
 \notag \int G(x,y)d\mu(y) \leq C \int d(x,y)^{2-n} d\mu(y)=C \int_0^\infty dt \ \mu \cur{ d^{2-n}>r } = C \int_0^\infty dr \ \mu \cur{ d<r^{\frac 1 {2-n}} }=\\
  = C (n-2)\int_0^\infty d\hat r  \ \hat r^{1-n}\ \mu \cur{ d<\hat r  } = C\int_0^\infty \frac {dr}{r} \frac 1 {r^{n-2}} \mu(\B r x)\, .
 \end{gather}
Let $1<q<\infty$ and $q'$ be its conjugate exponent. By (the continuous version of) Minkowski inequality
 \begin{gather}
  \norm{\int G(x,y)d\mu(y)}_{q'} \leq C\int_0^\infty \frac {dr}{r^{n-1}} \norm{\mu(\B r x)}_{L^{q'}(x)}\, .
 \end{gather}
We estimate very simply
\begin{gather}
 \ton{\mu\ton{\B r x} }^{\frac{q}{q-1}}= \ton{\mu\ton{\B r x} }^{\frac{1}{q-1}} \mu\ton{\B r x} \leq \min \ton{C r^s, \ \mu(M)}^{\frac 1 {q-1}}\mu\ton{\B r x}\, .
\end{gather}
Notice also that 
\begin{gather}
 \int_{M} dx \ \mu(\B r x) = \int_M dx \ \int_{\B r x} d\mu(y) = \int_M d\mu(y) \int_{\B r y} dy \leq C r^n \mu(M)\, .
\end{gather}
Putting these two together we get
\begin{gather}
 \norm{\mu}_{q'}^{q'}= \int_M dx \ \mu\ton{\B r x}^{\frac q{q-1}} \leq \min \ton{C r^s, \ \mu(M)}^{\frac 1 {q-1}}\int_M \mu\ton{\B r x}\leq \\
 \notag \leq C \min \ton{C\mu(M) r^{\frac s{q-1} +n}, \ \mu(M)^{q'} r^n}\, ,\\
 \notag \norm{\mu}_{q'}\leq C \min \ton{C\mu(M)^{\frac {q-1}{q}} r^{\frac s{q} +\frac{n(q-1)}{q}}, \ \mu(M) r^\frac{n(q-1)}{q}}
\end{gather}
Thus, if 
\begin{gather}\label{eq_cond_ns}
 2q-n<0\, , \qquad s+2q-n>0\, ,
\end{gather}
we get $\forall \kappa>0$:
\begin{gather}\label{e:k}
 \norm{\int G(x,y)d\mu(y)}_{q'} \leq C\mu(M)^{\frac{q-1}{q}} \int_0^\kappa r^{\frac {s-n}{q}+1} + C\mu(M) \int_\kappa^\infty  r^{1-\frac n q} =\\
\notag = C\mu(M)^{\frac{q-1}{q}} \int_0^\kappa r^{\frac {s+q-n}{q}} + C\mu(M) \int_\kappa^\infty  r^{\frac {q-n} q} = C \mu(M)^{\frac{q-1}{q}} \kappa^{\frac {s+2q-n}{q}} + C \mu(M) \kappa^{\frac {2q-n} q}
\end{gather}
By choosing the ``best'' k, the one which minimizes this last expression, which is $k=c \mu(M) ^{\frac 1 s}$, we get
\begin{gather}
 \norm{\int G(x,y)d\mu(y)}_{q'} \leq C \mu(M)^{\frac{q-1}{q}} \kappa^{\frac {s+2q-n}{q}} + C \mu(M) \kappa^{\frac {2q-n} q} =C \mu(M)^{1-\frac{n-2q}{sq}}\, .
\end{gather}

Choosing $q'=p$, we get the result. Notice that since $q$ must satisfy \eqref{eq_cond_ns}, i.e., $q\in \ton{\frac {n-s}{2},\frac n 2}$ we have that $p\in \ton{\frac{n}{n-2}, \frac{n-s}{n-s-2}}$.

Notice also that the exponent $1-\frac{n-2q}{sq}>0$ if $q>\frac{n-s}{2}$.
\end{proof}

\begin{remark}\label{rmk:log growth}
Note that Lemma \ref{lem:valto-adam} remains true also for $p=\frac{n-s}{n-2-s}$ if one assume instead of $\eqref{e:morrey growth}$ the strongest condition
  \begin{gather}
	\mu(\B r x)\leq C r^s(1+\abs{\log(r)})^{-2}\, , \qquad \forall x,r, \ \ 0\le s< n-2\, .
\end{gather}
 Indeed, under this assumption one can choose $\kappa=\infty$ in \eqref{e:k}.
	\end{remark}

\section{\texorpdfstring{Essential self-adjointness and $L^{p}$}{Lp}-operator cores for Schr\"odinger operators}\label{subsec:core}

The $L^{p}$-PP property first appeared, in a somewhat implicit form, in the seminal paper \cite{Ka} where T. Kato addressed the problem of the self-adjointness of Sch\"odinger operators with singular potentials in Euclidean spaces. Its validity on Riemannian manifolds was later systematically investigated with the aim of extending Kato's results to covariant Sch\"odinger operators with singular potentials on vector bundles; see e.g. \cite{BMS, Gu-JGEA} and references therein. More generally, a basic problem in the $L^{p}$ spectral theory of Schr\"odinger opeators is to understand under which conditions on the potential and on the underlying manifold, the space of smooth compactly supported functions is an $L^{p}$ core of the operator, namely, it is dense in the domain of its maximal realization. Again, in Euclidean space this is a classical result of Kato (see \cite{Ka2}), in the Riemannian case (under bounded geometry hypothesis) see \cite{Mi}, and in the
case of geodesically complete manifolds with lower bounded Ricci curvature see \cite[Appendix A] {GP}.

When considered in this framework, the results in our paper can be exploited to prove $L^{p}$ spectral properties for a class of Schr\"odinger operators on possibly incomplete manifolds. As we shall see in a moment, what we get looks relevant even for the Laplace-Beltrami operator.\smallskip

Assuming that $V \in L^{p}_{loc}(M)$ and $V \geq 0$, we define $(\Delta - V)_{\min,p}$ as the closure of $(\Delta-V)|_{C^{\infty}_{c}(M)}$ in $L^{p}(M)$. Furthermore, $(\Delta - V)_{\max,p}$ is defined as
\[
(\Delta - V)_{\max,p} u := (\Delta -V)u,\quad  \forall u \in \{w \in L^{p}(M) :  Vw\in L^{1}_{loc}(M) \text{ and }(\Delta -V)w \in L^{p}(M)\}.
\]

\begin{proposition}\label{coro:batu}
	Let $M$ satisfy the assumptions of either Proposition \ref{prop:LpPP cap} or Proposition \ref{prop:LpPPMink}. Assume that $p \in (1,+\infty)$, $V\geq 0$ and $V \in L^{p}_{loc}(M)$.
\begin{itemize}
\item [(a)] Then the operator $(\Delta - V)_{\min,p}$ generates a contraction semigroup in $L^{p}(M)$.
\item [(b)] Furthermore, $(\Delta - V)_{\min,p} = (\Delta - V)_{\max,p}$ or, in other words, $C^{\infty}_{c}(M)$ is an operator core for $(\Delta - V)_{\max,p}$.
\item [(c)] In particular, if $p=2$, then the operator $\Delta -V$ is essentially self-adjoint on $C^{\infty}_{c}(M)$.
\end{itemize}
\end{proposition}

\begin{proof}
 (c) Follows from \cite[Proposition 2.9 (a)]{Gu-JGEA} together with the $L^{2}$-PP property of Proposition   \ref{prop:LpPP cap} or Proposition \ref{prop:LpPPMink} above.
 
 (a) Follows from  \cite[Proposition 2.9 (b)]{Gu-JGEA} together with the $L^{p}$-PP property of Proposition   \ref{prop:LpPP cap} or Proposition \ref{prop:LpPPMink} above.
 
 (b) Note that (a) with $V \equiv 0$ tells us that $\Delta_{\min,q}$ generates a contraction semigroup in $L^{q}(M)$, where $q$ is the H\"older conjugate of $p$. Now, we can repeat the argument in the proof of \cite[Theorem 5]  {GP} to show that $\Delta_{\min,p}=\Delta_{\max,p}$. Having established (a) and having $\Delta_{\min,p}=\Delta_{\max,p}$ at our disposal, it remains to remember the $L^{p}$-PP property of Proposition   \ref{prop:LpPP cap} or Proposition \ref{prop:LpPPMink}  and observe that the argument of Sections 2 and 3 of \cite{Mi} are applicable in our context. This proves property (b).
\end{proof}

	A quick comparison with the current literature permits to put our result in perspective and highlight its novelties.\smallskip
	
	In \cite{CdV}, the essential self-adjointness of the Laplace-Beltrami operator is proved on a punctured manifold and in \cite{Ma} the study is extended to incomplete Riemannian manifolds of the form $M= N \setminus K$ where the removed set $K$ is a smooth closed submanifold $K$ of co-dimension greater than $3$. In the very recent \cite{HMS-preprint} (see also \cite{HMS-journal}) the case of singular removed sets $K$ (in the general setting of metric measure spaces) is considered. The authors show that for a closed (possibly non-compact)  set $K$ with  $\dim_{H}(K)<n-2p$,  the space  $C^{\infty}_{c}(M)$ is an $L^{p}$-operator core of the Laplacian on $M$ (and, hence, when $p=2$ the Laplacian is essentially self-adjoint). It is also shown that this property fails if the singular set $K$ has zero-Riemannian measure and Hausdorff dimension $\dim_{H}(K)>n-2p$. The treshold case $n-2p$ is included in the picture by using a suitably log-correction in the Hausdorff measure. Thus, in a rather precise sense, our result completes those in \cite{HMS-preprint}. Indeed, for a fixed dimension, in that paper the $L^{p}$-operator core property for the Laplacian is established for small values of $p$ whereas we  are able to deal with large values $p \geq 2$. It should be also pointed out that combining \cite{HMS-preprint} with our Proposition \ref{coro:batu}  yields that, as somehow expected, $L^{p}$-positivity preservation and  $L^{p}$-operator core properties for the Laplacian are distinct concepts.

In a different but related direction, in \cite{Ma, GM} the authors investigate spectral properties of the Gaffney Laplacian $\Delta_G : W^{1,2}(M) \to L^{2}(M)$ on incomplete manifolds, obtaining conditions for its self-adjointness in terms of probabilistic properties like parabolicity and stochastic completeness and related to the Minkowski content of the removed part, somehow in the spirit of our Proposition \ref{prop:LpPPMink}.

	In the context of Schr\"odinger operators, under our assumptions on $V$ and for $p \not=2$, Proposition \ref{coro:batu} is new even in the case when $M$ is geodesically complete. Indeed, note that properties (a) and (b) for a geodesically complete manifold were also obtained in \cite{Mi2019}, but under the more stringent hypothesis $0\leq V \in L^{\infty}_{loc}(M)$. When $M$ is geodesically complete, Proposition \ref{coro:batu} (iii) is contained in (various) known self-adjointness results; see for instance the main result of \cite{GuPo}, where the authors allow $V$ to have a negative part belonging to the Kato class on $M$.

   Schr\"odinger operators on geodesically incomplete manifolds were also considered in  \cite{MiTu}  where the authors
require no assumptions on the geometry of $M$, but ask for a controlled behaviour of the potential $V$ near the Cauchy boundary of $M$.

\section{Acknowledgments} We are indebted to Andrea Bonfiglioli and Ermanno Lanconelli for clarifying some points about 
their article \cite{BL} and to Peter Sj\"ogren for sharing with us some useful comments about the Brelot-Hervé theory of subharmonic functions and the smoothing approximation procedure. We also thanks Jun Masamune for the relevant references \cite{HMS-journal, HMS-preprint} and to Ognien Milatovic for a  careful reading of the first draft of the paper and for several suggestions that improved both the presentation and the result of Section \ref{subsec:core}. Finally, we are very grateful to Batu G\"uneysu for pointing out to us the BMS conjecture, and for several fruitful discussions on the topic during the last years.

Partially supported by INdAM-GNAMPA.

\bibliographystyle{aomalpha}
\bibliography{PVV_LpPP}

\end{document}